\documentclass[11pt,letterpaper]{amsart}
\usepackage{url,amsmath,amsthm, amssymb}
\usepackage[margin=1in]{geometry}
\usepackage[pdfencoding=auto, psdextra, colorlinks=true, linkcolor = blue, urlcolor=blue, citecolor=blue, bookmarksdepth=10]{hyperref}
\usepackage{bookmark}
\usepackage[all]{mathtools}
\usepackage{mathrsfs}
\usepackage{mathdots}
\usepackage{amsxtra,amscd}
\usepackage{amsrefs}
\usepackage{enumerate, enumitem}
\usepackage{arydshln}
\usepackage{multirow}
\usepackage{bm} 
\usepackage[all,arc]{xy}

 \usepackage{amssymb,amsrefs}
\usepackage{amsmath}
\usepackage{amscd,latexsym}
\usepackage{bookmark}
 \usepackage{amssymb,amsfonts,bm}
\usepackage[all,arc]{xy}
\usepackage{enumerate}
\usepackage{mathrsfs}
\usepackage{amscd}
\usepackage{tikz}
\usepackage{tikz-cd}
\usepackage{amsmath}
\usepackage{latexsym}
\usepackage{mathdots}
\usepackage{amsrefs}
\usepackage{graphicx}
\usepackage{mathtools}
\usepackage{leftidx}
\usepackage{tensor}
\usepackage{dsfont}
\usepackage{tikz-cd} 
\usepackage[new]{old-arrows}
\usepackage{xcolor}
\usepackage{bbm}
\usepackage{enumitem}

\theoremstyle{plain}
\newtheorem{theorem}{Theorem}[section]
\newtheorem{corollary}[theorem]{Corollary}

\newtheorem{lemma}[theorem]{Lemma}
\newtheorem{proposition}[theorem]{Proposition}
\numberwithin{equation}{section}

\theoremstyle{definition}

\theoremstyle{remark}
\newtheorem{remark}[theorem]{Remark}

\newtheorem*{acknowledgments}{Acknowledgments}

\begin{document}

\title{The local period integrals and essential vectors}

\author[]{Yeongseong Jo}
\address{Department of Mathematics, The University of Iowa, Iowa City, IA 52242, USA}
\curraddr{Department of Mathematics and Statistics, The University of Maine, Orono, ME 04469, USA}
\email{\href{mailto:jo.59@buckeyemail.osu.edu}{jo.59@buckeyemail.osu.edu}}


\begin{abstract}
By applying the formula for essential Whittaker functions established by Matringe and Miyauchi, we study five integral representations for irreducible admissible generic representations of ${\rm GL}_n$ over $p$-adic fields. In each case, we show that the integrals achieve local formal $L$-functions defined by Langlands parameters, when the test vector is 
associated to the new form. We give the relation between local periods involving essential Whittaker functions and special values of formal $L$-factors at $s=1$ 
for certain distinguished or unitary representations. The period integrals are also served as standard non-zero distinguished forms. 
\end{abstract}

\maketitle


\section{Introduction}
\label{Intro}

Let $F$ be a non-archimedean local field of characteristic zero with the ring of integers $\mathcal{O}$ and residue field of cardinality $q$. The essential vector which is known as the local new form plays an important role in the theory of automorphic forms. The history of the essential vector goes back to at least Casselman, where he established a theory of new forms for ${\rm GL}_2(F)$ \cite{Casselman}. In this paper, by applying the essential Whittaker functions (also called newforms), we study the test vector problem and the non-vanishing of local periods for five integrals representing Rankin-Selberg model \cite{Bern,FLO,Zhang}, Flicker-Rallis model \cite{AM,AKT,Fli91,Zhang}, Jacquet-Shalika model \cite{JaRa,Jo20,Ma14JNT}, Friedberg-Jacquet model \cite{FriJac,JaRa,Ma14JNT,MA15,MA17}, and Bump-Ginzburg model \cite{BuGi92,Kaplan17,KaYa,Yamana}. The local $L$-functions of $GL_n(F)$ that are associated to these integrals include the tensor product $L$-factor of ${\rm GL}_n(F) \times {\rm GL}_n(F)$, the Asai $L$-factor, the exterior square $L$-factor, the Bump-Friedberg $L$-factor, and the symmetric square $L$-factor.

\par 
Let $\pi$ and $\sigma$ be irreducible admissible generic representation of ${\rm GL}_n(F)$ and ${\rm GL}_m(F)$. For simplicity, we only illustrate the first fundamental problem for the Rankin-Selberg integrals of ${\rm GL}_n(F) \times {\rm GL}_n(F)$ in this induction. We fix an additive character $\psi$ of conductor $\mathcal{O}$. Given a pair of Whittaker functions $W_{\pi} \in \mathcal{W}(\pi,\psi)$ and $W_{\sigma} \in \mathcal{W}(\sigma,\psi^{-1})$ and given a Schwartz–Bruhat function $\Phi \in \mathcal{S}(F^n)$, we define the local Rankin-Selberg integral \cite{JPSS3};
\begin{equation}
\label{sameRS}
  \Psi(s,W_{\pi},W_{\sigma},\Phi)=\int_{N_n(F) \backslash {\rm GL}_n(F)} W_{\pi}(g) W_{\sigma}(g) \Phi(e_ng) |\mathrm{det}(g)|^s dg,
\end{equation}
where $e_n=(0,\dotsm,0,1)$ and $N_n(F)$ is the group of unipotent matrices. It is absolute convergent for $\mathrm{Re}(s)$ sufficiently large
and the collection of such integrals generates a fractional ideal $\mathcal{I}(\pi \times \sigma)=\mathbb{C}[q^s,q^{-s}]L(s,\pi \times \sigma)$ of $\mathbb{C}(q^{-s})$. We choose the normalized generator of the form  $L(s,\pi \times \sigma)=P(q^{-s})^{-1}$, where $P(X) \in \mathbb{C}[X]$ is a polynomial with $P(0)=1$. We define a map 
\[
\mathcal{W}(\pi,\psi) \otimes \mathcal{W}(\sigma,\psi^{-1}) \otimes \mathcal{S}(F^n) \rightarrow \mathbb{C}(q^{-s}) \quad \text{by} \quad W_{\pi} \otimes W_{\sigma} \otimes \Phi \mapsto \Psi(s,W_{\pi},W_{\sigma},\Phi).
\]
Then there exists an element in $\mathcal{I}(\pi \times \sigma)$ that transports to $L(s,\pi \times \sigma)$. It is {\it a priori} a finite sum of the form $\sum_i \Psi(s,W_{\pi,i},W_{\sigma,i},\Phi_i)$. The {\it strong test vector problem} is to determine the existence of a triple of pure tensors $W_{\pi} \in \mathcal{W}(\pi,\psi)$, $W_{\sigma} \in \mathcal{W}(\sigma,\psi^{-1})$, and $\Phi \in \mathcal{S}(F^n)$ which yields $\Psi(s,W_{\pi},W_{\sigma},\Phi)=L(s,\pi \times \sigma)$ (cf. \cite[\S 1.6-1]{Cogdell}). In this scenario, $(W_{\pi},W_{\sigma},\Phi)$ is called a {\it strong test vector}.

\par
We write the Godement–Jacquet standard $L$-factors for $\pi$ and $\sigma$ \cite{Jacquet} as
\[
 L(s,\pi)=\prod^r_{i=1} (1-\alpha_iq^{-s})^{-1} \quad \text{and} \quad  L(s,\sigma)=\prod^p_{j=1} (1-\beta_jq^{-s})^{-1},
\]
respectively. We construct unramified standard modules $\pi_{ur}$ and $\sigma_{ur}$ attached to the Langlands parameter $\{ \alpha_i \}_{i=1}^r$ and $\{ \beta_j \}_{j=1}^p$. We define the {\it formal $L$-factor} for a pair $(\pi,\sigma)$ by
\[
  \prod^r_{i=1} \prod^p_{j=1} (1-\alpha_i\beta_jq^{-s})^{-1} 
\]
which coincides with $L(s,\pi_{ur} \times \sigma_{ur})$ (see Proposition \ref{Langlands-RS}). In general, $L(s,\pi_{ur} \times \sigma_{ur})^{-1}$  divides $L(s,\pi \times \sigma)^{-1}$ in $\mathbb{C}[q^{\pm s}]$ (see Corollary \ref{RS-divisibility}). The purpose of this paper is to resolve the test vector problem for various formal $L$-factors, especially, in the case of ramified representations through local means. We call this problem the {\it weak test vector problem} and an associated triple $(W_{\pi},W_{\sigma},\Phi)$ a {\it weak test vector}. The reader consult later sections for details about the unexplained notations in the exact statements below.

\begin{theorem}[Weak Test Vector Problems] \,
\begin{enumerate}[label=$(\mathrm{\roman*})$]
\item\label{TEST-1}$[$Rankin-Selberg $L$-factors$]$ Let $\pi$ and $\sigma$ be irreducible admissible generic representations of ${\rm GL}_n(F)$. Then we have
\[
 L(s,\pi_{ur} \times \sigma_{ur})=\int_{N_n \backslash {\rm GL}_n(F)} W^{\circ}_{\pi}(g)W^{\circ}_{\sigma}(g) \Phi_c(e_ng) |\mathrm{det}(g)|^s dg.
\]
\item\label{TEST-2}$[$Asai $L$-factors$]$ Let $\pi$ be irreducible admissible generic representations of ${\rm GL}_n(E)$ and $E$ a quadratic extension of $F$. Then we have 
\[
L(s,\pi_{ur},As)=\int_{N_n \backslash {\rm GL}_n(F)} W^{\circ}_{\pi}(g) \Phi_c(e_ng) |\mathrm{det}(g)|^s dg.
\]
\item\label{TEST-3}$[$Symmetric square $L$-factors$]$ Let $\pi$ be irreducible admissible generic representations of ${\rm GL}_m(F)$. Then we have 
\[
  \mathcal{L}(s,\pi_{ur},\mathrm{Sym}^2)=\int_{\mathscr{Z}_mN_m \backslash {\rm GL}_m(F)} W^{\circ}_{\pi} (g) W^{{e}^{\circ}}_{\theta_m^{\psi}} (g) f^{K_1(\mathfrak{p}^c)}_{(2s-1)}(g) dg.
\]
\end{enumerate}
\end{theorem}

The weak test vector problem has been resolved for formal exterior square $L$-factors $L(s,\pi_{ur},\wedge^2)$ and formal Bump-Friedberg $L$-factors $L(s_1,\pi_{ur})L(s_2,\pi_{ur},\wedge^2)$ by Miyauchi and Yamauchi \cite{MY}.  For different rank cases $n > m$, Booker, Krishnamurthy, and Lee \cite{BKL} work out formal Rankin-Selberg $L$-factors $L(s,\pi_{ur} \times \sigma_{ur})$. A major shortcoming for treating Rankin-Selberg integrals of different rank groups is that Schwartz–Bruhat functions $g \mapsto \Phi(e_mg)$ defined on ${\rm GL}_m(F)$ as seen in \eqref{sameRS} is no longer present so one is not able to control the last row of the element $g$ in the smaller group ${\rm GL}_m(F)$.  In order to incorporate the absence, they develop a novel approach of unipotent averaging to modify the essential Whittaker function defined on the bigger group ${\rm GL}_n(F)$. This feature does not happen to our circumstance and our machinery can be applied uniformly to other ${\rm GL}_n(F)$-type formal local $L$-functions.

  \par
The construction of the strong test vector for Rankin-Selberg $L$-factors $L(s,\pi \times \sigma)$ with $n \geq m$
is carried out as a main part of the Ph.D thesis by Kim \cite[Theorem 2.1.1]{Kim} supervised by Cogdell, when at least one of two representations $\pi$
and $\sigma$ is unramified. The strong test vector problem is pursued by  Kurinczuk and Matringe \cite{KuMa} for a pair of discrete series representations $\pi$ and $\sigma$,
though the space $\mathcal{W}(\sigma,\psi^{-1})$ is enlarged to the Whittaker model for the standard module associated to $\sigma$. 
Recently, Humphries \cite{Humphries} complete the case of standard $L$-factors $L(s,\pi)$.
There has been renowned work \cite{AKMSS} involved in determining an explicit Whittaker function and a characteristic function for local Asai $L$-functions $L(s,\pi,As)$, but again the strong test vector is the Paskunas–Stevens's cut-off Whittaker function relying on the type theory of Bushnell-Kutzko. 
Until this moment, it is unclear that pure tensors arising from newforms solve strong test vector problems even for Rankin-Selberg or Asai $L$-factors and  what the concrete formulas for strong test vectors would look like in the general context. 
Besides general linear groups, adapting newforms as previously discussed, parallel strong test vector problems have been investigated by Miyauchi \cite{Miy18} for $L$-factors of unramified ${\rm U}(2,1)$, and by Roberts and Schmidt \cite[Theorem 7.5.4]{RobertsShmidt} for $L$-factors of ${\rm GSp}(4)$ attached to irreducible generic admissible representations.

\par
Let $H$ be a closed subgroup of ${\rm GL}_n(F)$ and $\chi$ a character of $H$. Let $\mathrm{Hom}_{H}(\pi,\chi)$ be the space of the linear forms $\Lambda : V \rightarrow \mathbb{C}_{\chi}$ such that $\Lambda(\pi(g)v)=\chi(g)\Lambda(v)$ for $g \in H$ and $v \in V$. We say that $\pi$ is $(H,\chi)$-distinguished if $\mathrm{Hom}_{H}(\pi,\chi) \neq 0$. In particular, if $\chi={\bf 1}_H$ is a trivial character of $H$, $\pi$ is called $H$-distinguished. In order to present what to expect, we shall attempt to elaborate the second main problem for by-now well-known as Flicker-Rallis periods \cite[\S 3.2]{Zhang}, which first appeared in \cite[\S 3]{Fil88}. Let $E \slash F$ be a quadratic extension. We specify the case $G={\rm GL}_n(E)$ and $H={\rm GL}_n(F)$. Let $\pi$ be an irreducible admissible representation of $G$. One of important parts
for $H$-distinguished generic representations is that an explicit $H$-invariant linear functional in the space $ \mathrm{Hom}_{H}(\pi|_{H},{\bf 1}_{H})$ can be realized as an integral over $N_n(F) \backslash P_n(F) \simeq N_{n-1}(F) \backslash {\rm GL}_{n-1}(F)$  \citelist{\cite{AM}*{Theorem 1.1} \cite{AKT}*{Corollary 1.2}}, where $P_n(F)$ denotes the mirabolic subgroup consisting of invertible matrices whose last row equals to $e_n$. Up to multiplication by scalars, the one-dimensionality of the space ${\rm Hom}_{P_n(F)}(\pi,{\bf 1}_{P_n(F)})$, as a consequence of \cite[Proposition 11]{Fli91} and \cite[Proposition 2.3]{Ma14} (cf. \citelist{\cite{AKT}*{Theorem 1.1} \cite{Ma10}*{Proposition 1.1}}), ensures that such a unique $H$-invariant form on the Whittaker model $\mathcal{W}(\pi,\psi_E)$ can be precisely written down as
\[
  \Lambda^{FR}(W_{\pi})=\int_{N_n(F) \backslash P_n(F)} W_{\pi}(p) dp=\int_{N_{n-1}(F) \backslash {\rm GL}_{n-1}(F)} W_{\pi} \begin{pmatrix} g & \\ & 1 \end{pmatrix}  dg.
\]
It is known to be convergent under the unitarity assumption \cite[p.306]{Fil88}, but even in the non-unitary case we could make sense of the above integral when $\pi$ is distinguished with respect to $H$ \cite[Remark 2]{AM}. This extension property from $P_n(F)$ to $H$-invariant forms was brought to Bernstein's attention in the framework of Rankin-Selberg periods \cite[Proposition 5.3]{Bern}. 

\par
The second aim of this paper is to formulate a relation between the values of local period integrals at essential Whittaker functions and special values of formal $L$-factors, or their ratios to Tate's $L$-factors at $s=1$, other than the Flicker-Rallis period. As a byproduct we provide a nice application to the non-vanishing of local period integrals. In addition we give a constructive and purely local proof of the existence of some non-zero 
invariant functionals, which reflects on Bernstein's well-known principle.  As before, the reader is advised to consult the following sections for undefined terminologies in the central result of this paper below.

\begin{theorem}[Local Periods] \,
\begin{enumerate}[label=$(\mathrm{\roman*})$]
\item\label{PERIOD-1}$[$Rankin-Selberg Periods$]$ Let $\pi$ and $\sigma$ be irreducible admissible generic representations of ${\rm GL}_n(F)$ such that $\pi \otimes \sigma$ is distinguished with respect to ${\rm GL}_n(F)$. Then we have
\[
 \int_{N_n \backslash P_n} W^{\circ}_{\pi}(p) W^{\circ}_{\pi^{\vee}}(p) dp=
  \begin{cases}
  L(1,\pi_{ur} \times \pi^{\vee}_{ur})& \text{if $\pi$ is ramified,} \\
  \dfrac{L(1,\pi \times \pi^{\vee})}{L(n,{\bf 1}_{F^{\times}})}& \text{otherwise.} 
  \end{cases}
\]
\item\label{PERIOD-2}$[$Jacquet-Shalika Periods$]$ 
Let $\pi$ be a irreducible admissible generic representation of ${\rm GL}_{2n}(F)$ which is distinguished with respect to $(S_{2n},\Theta)$. Then we have
\[
\int_{N_n \backslash P_n} \int_{\mathcal{N}_n \backslash \mathcal{M}_n} W^{\circ}_{\pi} \left( \sigma_{2n} \begin{pmatrix} I_n & X \\ & I_n  \end{pmatrix} \begin{pmatrix} p & \\ & p \end{pmatrix} \right) \psi^{-1}({\rm Tr}(X)) dX dp
=
 \begin{cases}
  L(1,\pi_{ur},\wedge^2)& \text{if $\pi$ is ramified,} \\
  \dfrac{L(1,\pi,\wedge^2)}{L(n,{\bf 1}_{F^{\times}} )}& \text{otherwise.}
  \end{cases}
\]
Moreover, if $\pi$ is unitary, then the integral realizes the unique Jacquet-Shalika functional in the space $\mathrm{Hom}_{S_{2n}}(\pi, \Theta)$.
\item\label{PERIOD-3}$[$Friedberg-Jacquet (Linear) Periods$]$ Let $\pi$ be an irreducible admissible generic representations of ${\rm GL}_{2n}(F)$ which is distinguished with respect to $H_{2n}$. Then we have
\[
 \int_{(N_{2n} \cap H_{2n}) \backslash (P_{2n} \cap H_{2n}) }W^{\circ}_{\pi}(p) dp=
  \begin{cases}
    L(1/2,\pi_{ur})L(1,\pi_{ur},\wedge^2)
   & \text{if $\pi$ is ramified,} \\
 \dfrac{L(1/2,\pi)L(1,\pi,\wedge^2)}{L(n,{\bf 1}_{F^{\times}})}  & \text{otherwise.} 
  \end{cases}
\]
Moreover, if $\pi$ is unitary, then the integral realizes the unique Friedberg-Jacquet functional in the space $\mathrm{Hom}_{H_{2n}}(\pi, {\bf 1}_{H_{2n}})$.
\item\label{PERIOD-4}$[$Bump-Ginzburg Periods$]$ Let $\pi$ be an irreducible admissible unitary generic representation of  ${\rm GL}_m(F)$. Then we have
\[
\int_{\mathscr{Z}_mN_m \backslash G_m} W^{\circ}_{\pi} (g) W^{{e}^{\circ}}_{\theta_m^{\psi}} ({\bf s}(g)) f^{K_1(\mathfrak{p}^c)}_{(1)}({\bf s}(g)) dg=\mathcal{L}(1,\pi_{ur},\mathrm{Sym}^2).
  \]
Moreover, if $\pi$ is $\theta$-distinguished, the integral realizes the unique $G_m$-invariant trilinear form in the space 
  $\mathrm{Hom}_{G_m}(\pi \otimes \theta^{\psi}_m \otimes \theta^{{\psi}^{-1}}_m,{\bf 1}_{G_m})$.
\end{enumerate}
\end{theorem}

Our result can be considered as weak test vector problems for period integrals and it is
greatly influenced by the work of Anandavardhanan and Matringe \cite{AM} in which the analogue expression for Flicker-Rallis period integrals was established. A similar formalism for Jacquet-Shalika periods appeared in \cite{Grobner}, although Grobner only considered unitary representations as a local component of 
a cuspidal automorphic representation, which is sufficient to their global applications. The preprint \cite{Grobner} was built upon generous explanation of N. Matringe \cite[Abstract]{Grobner} and we take this occasion to announce formally and extend his result to more general settings.

\par
The distinction condition is superfluous if our representations $\pi$ (and if necessary $\sigma$) are assumed to be unitary. However, with the exception of Bump-Ginzburg period cases, we do impose the distinction assumption in non-unitary generic cases. The main rationale to rule out $\theta$-distinguished representations is that the equality $L(s,\pi,\mathrm{Sym}^2)=\mathcal{L}(s,\pi,\mathrm{Sym}^2)$ is solely available for $n=2$ \cite{Jo21}, where $L(s,\pi,\mathrm{Sym}^2)$ is defined by the normalized generator of a fractional ideal as explained in \cite[Definition 3.12]{Yamana}.
Therefore we cannot guarantee the holomorphy of Bump-Ginzburg integrals at $s=1$  (see Remark \ref{rmk2}). We hope that we tackle this issue for representations of higher ranked groups in our work in progress. 
Unlike other period integrals, a keen reader may notice that our expression for Bump and Ginzburg period integrals is taken over general linear groups, which admits trilinear forms in the same space $\mathrm{Hom}_{G_m}(\pi \otimes \theta^{\psi}_m \otimes I_{\psi}(1,\omega^{-1}_{\pi}),{\bf 1}_{G_m})$ as the one in \cite[Theorem 2.14.(2)]{Yamana}. But we can still manipulate period integrals in such a way that they are integrated over mirabolic subgroups (see Remark \ref{rmk1}).

\par
In contrast to Rankin-Selberg and Flicker-Rallis periods, the unitary hypothesis is required for the corresponding results to execute aforementioned Bernstein's extension philosophy \cite[Proposition 5.3]{Bern}, because we do not have at most one-dimensionality of larger spaces $\mathrm{Hom}_{S_{2n} \cap P_{2n}}(\pi, \Theta)$, $\mathrm{Hom}_{H_{2n} \cap P_{2n}}(\pi, {\bf 1}_{H_{2n} \cap P_{2n}})$, and $\mathrm{Hom}_{G_m}(\pi \otimes \theta^{\psi}_m \otimes I_{\psi}(1,\omega^{-1}_{\pi}),{\bf 1}_{G_m})$ for non-unitary generic cases. We anticipate that the unitarity assumption is unnecessary. Thankfully, somehow this stronger uniqueness is only relevant for beautiful applications to the characterization of the occurrence of (exceptional) poles of local $L$-functions in terms of the existence of the unique non-zero invariant form. There has been a flurry of work on exploring this subject extensively by Matringe for Rankin-Selberg $L$-factor \cite[Proposition 4.6]{MA15}, Asai $L$-factor \cite[Theorem 3.1]{Ma10}, and Bump-Friedberg $L$-factor \cite[Proposition 4.12]{MA15}, by the author for exterior square $L$-factor \cite[Lemma 3.2]{Jo20}, and by Yamana for symmetric square $L$-factor \cite[Theorem 3.17]{Yamana}. In practice, 
the connection between poles of local $L$-function and distinctions has to be proven beforehand to establish inductive relations of $L$-factors. Over the course of the discussion, it is raised by Kaplan \cite[Remark 4.18]{Kaplan17} whether $\theta$-distinguished discrete series will be self-dual or not in the frame of positive characteristic fields.
It is our belief that Bump-Ginzburg periods can be transferred to $\theta$-distinguished representations and ultimately the distinction sheds some lights on Kaplan's question by detecting (exceptional) poles of $L(s,\pi,\mathrm{Sym}^2)$. We plan to turn to his observation in the near future.


\par
Our proof takes the chief ingredient from the formula for essential Whittaker functions associated to essential vectors on the sub-torus \cite{JacquetCorrection,JP-SS81}.  The key formulation is constructed independently by Matringe \cite[Corollary 3.2]{Ma13} and Miyauchi \cite[Theorem 4.1]{Miy}, which generalizes Shintani's method for spherical representations. In the spirit of \cite{Humphries-Preprint}, it would be interesting to find analogous weak test vectors attached to essential Whittaker functions for Archimedean ${\rm GL}_n$-type local Zeta integrals \cite{HJ21} (slightly different problems concerning cohomological vectors have been suggested in \cite[\S 1.6-2]{Cogdell}).  

\par
The brief structure of this paper is the following. Section \ref{NV} contains preliminaries, including the review of Langlads parameters, essential vector and Whittaker functions, and the theory of the local standard $L$-functions. For the remaining sections, the first half of each sections deals with unitary generic cases and the parallel construction for non-generic cases is presented in the second half. The theory of local Rankin-Selberg integrals occupies Section \ref{RS}. This serves to overview our methodology that is repeated throughout the paper. In Section \ref{Asai}, we mainly quote crucial results about Flicker-Rallis period integrals and Asai $L$-factors from \cite{AM}. 
We discuss local exterior square $L$-functions and their related periods by Jacquet and Shalika in Section \ref{JSperiod} - \ref{JSS2n} and by Bump and Friedberg in Section \ref{FJperiod} - \ref{BJH2n}.
Section \ref{BG} is devoted to finding test vectors for the Bump-Ginzburg period and symmetric square $L$-factors associated to Langlands parameters.

\section{Essential Vectors}
\label{NV}

Let $F$ be a non-archimedean local field of characteristic zero, $\mathcal{O}$ the ring of integers of $F$,  $\mathfrak{p}$ a unique prime ideal of $\mathcal{O}$, and $\varpi$ a fixed choice of a uniformizer of the prime ideal. Let $q$ denote the cardinality of its residue field. Let $|\cdot|$ denote the standard normalization so that $|\varpi|=q^{-1}$. The character of ${\rm GL}_n(F)$ given by $g \mapsto |\mathrm{det}(g)|$ is denoted by $\nu$.

\par
We denote by $\mathcal{A}_F(n)$ the set of the equivalent classes of all irreducible admissible complex valued representations of ${\rm GL}_n(F)$ and by $\mathcal{A}_F$ the union $\cup_n \mathcal{A}_F(n)$. We recall that $\Delta \in \mathcal{A}_F(n)$ is called {\it quasi-square integrable} if, after twisting by a character, it is {\it square integrable} (also called {\it discrete series}).
 The quasi-square integrable representations of ${\rm GL}_n(F)$ have been classified by Bernstein and Zelevinsky. According to \cite[Theorem 9.3]{Zel80}, such a representation $\Delta$ is the unique irreducible quotient of the form
\[
 {\rm Ind}^{{\rm GL}_n(F)}_{\rm Q}(\rho\nu^{1-\ell} \otimes \rho\nu^{2-\ell} \otimes \dotsm \otimes \rho)
\]
where the induction is normalized parabolic induction from the standard parabolic subgroup $\rm Q$ attached to the partition $(a,a,\dotsm,a)$ of $n=a\ell$ and $\rho \in \mathcal{A}_F(a)$ is supercuspidal.
Briefly $\Delta$ is parametrized by $\rho$ and we denote by $\Delta(\rho)=[\rho\nu^{1-\ell},\rho\nu^{2-\ell}, \dotsm,\rho]$ such a quotient. Further $\Delta$ is square integrable if and only if the supercuspidal representation $\rho\nu^{-\frac{\ell-1}{2}}$ is unitary. Also $[\rho\nu^{1-\ell},\rho\nu^{2-\ell}, \dotsm,\rho]$ is said to be a {\it segment}. 

\par
Let $\rm P$ be a standard upper parabolic subgroup of ${\rm GL}_n(F)$ of type $(n_1,n_2,\dotsm,n_t)$ with $n=n_1+n_2+\dotsm +n_t$. 
We write $\delta_{\rm P}$ for the modulus character of $\rm P$.
For each $1 \leq i \leq t$, let $\Delta^0_i \in \mathcal{A}_F(n_i)$ be a square integrable representation. Let $(s_1,s_2,\dotsm,s_t)$ be a sequence of ordered real numbers so that $s_1 \geq s_2 \geq \dotsm \geq s_t$. The normalized induced representation
\[
 {\rm Ind}^{{\rm GL}_n(F)}_{\rm P}(\Delta^0_1\nu^{s_1} \otimes \Delta^0_2\nu^{s_2}  \otimes \dotsm \otimes \Delta^0_t\nu^{s_t} )
\]
is called a {\it standard module} of ${\rm GL}_n(F)$. If $\pi \in \mathcal{A}_F(n)$, it is well-known that it is realized as the unique irreducible quotient, called the {\it Langlands quotient}, of some standard module of ${\rm GL}_n(F)$.

\par
We set $G_n={\rm GL}_n(F)$. Let $B_n$ be the Borel subgroup of upper triangular matrices in $G_n$, and $N_n$ its unipotent radical. The maximal torus of $G_n$ consisting of all diagonal matrices is denoted by $A_n$. We write $Z_n$ to denote the center consisting of scalar matrices. We define $P_n$ the mirabolic subgroup of $G_n$ given by
\[
 P_n= \left\{ \begin{pmatrix} g_{n-1}  & x \\ & 1 \end{pmatrix} \; \middle| \; g_{n-1} \in G_{n-1}, x \in F^{n-1} \right\}.
\]
We denote by $U_n$ the unipotent radical of $P_n$ and we put $P_{n-1,1}=Z_nP_n$. As a group, $P_n$ possess a structure of a semi-direct product $P_n = G_{n-1} \ltimes U_n$.

\par
We fix a non-trivial unramified character $\psi$, so $\psi(\mathcal{O})=1$ but $\psi(\varpi^{-1}\mathcal{O})\neq 1$. We let $\psi$ denote by abuse of notation  the character of $N_n$ defined by
\[
  \psi(n)=\psi( \sum_{i=1}^{n-1} n_{i,i+1} ), \quad n=(n_{i,j}) \in N_n.
\]
An admissible representation $\pi$ of $G_n$  is said to be {\it generic} if $\pi$ is $(N_n,\psi)$-distinguished, namely,
\[
\mathrm{Hom}_{G_n}(V, \mathrm{Ind}^{G_n}_{N_n}(\psi))\simeq \mathrm{Hom}_{N_n}(V,\mathbb{C}_{\psi}) \neq 0. 
\]
Then there exist a non-zero linear form called a {\it Whittaker functional} $\lambda : V \rightarrow \mathbb{C}_{\psi}$ satisfying $\lambda(\pi(n)v)=\psi(n)\lambda(v)$ for $n \in N_n$ and $v \in V$. It is known that for either an irreducible representation \cite[Theorem A]{GK} or a standard module $\pi$ \cite[\S 1.2]{JaSh83}, the space of such a functional is of dimension $1$. Let $\mathcal{W}(\pi,\psi)$ denote the {\it Whittaker model} via the space of functions $(W_{\pi})_v$ on $G_n$ defined by $(W_{\pi})_v(g)=\lambda(\pi(g)v)$ for $v \in V$.

\par
For a ramified generic representation $\pi \in \mathcal{A}_F(n)$, the unramified standard module $\pi_{ur} \in  \mathcal{A}_F(p)$ with $p < n$ \cite[Definition 1.3]{Ma13} is associated to $\pi$ as follows. Let 
\[
  \pi={\rm Ind}^{{\rm GL}_n(F)}_{\rm P}(\Delta_1(\rho_1) \otimes \Delta_2(\rho_2) \otimes \dotsm \otimes \Delta_t(\rho_t) ) \in \mathcal{A}_F(n)
\]
be a ramified generic representation. If $\mathfrak{X}_{ur}(\pi):=\{  \rho_i  :  \rho_i \; \text{is an unramified character of $F^{\times}$} \}$, the subset of $\{ \rho_1,\rho_2,\dotsm, \rho_p \}$, is not empty, we denote by $\alpha_1,\alpha_2,\dotsm, \alpha_p$ the ordered (maybe equal) elements of $\mathfrak{X}_{ur}(\pi)$ satisfying $\mathrm{Re}(\alpha_i) \geq \mathrm{Re}(\alpha_{i+1})$. We define $\pi_{ur}$ as the trivial representation $\textbf{1}$ when $\mathfrak{X}_{ur}(\pi)$ is empty and the unramified standard module
\[
  \pi_{ur}:={\rm Ind}^{{\rm GL}_p(F)}_{B_p}(\alpha_1 \otimes \alpha_2 \otimes \dotsm \otimes \alpha_p)
\]
otherwise. A significance for deploying $\pi_{ur}$ is that the shape of the standard $L$-factor for the original representation $\pi$ is encoded in the standard $L$-factor for $\pi_{ur}$   \cite{Jacquet}.

\begin{theorem}[Godement and Jacquet]
\label{G-J} 
Let $\pi \in \mathcal{A}_F(n)$ be a generic representation. Then
\[
  L(s,\pi)=L(s,\pi_{ur})=\prod_{i=1}^p (1-\alpha_i(\varpi)q^{-s})^{-1}.
\]
\end{theorem}
We call the set $\{\alpha_i(\varpi) \}_{i=1}^p$ the {\it Langlands parameter} of $\pi$. If $\pi$ is unramified, they agree with the usual Satake parameters. The standard local $L$-function $L(s,\pi_{ur})$ is nothing but employed to define what is called the {\it naive Rankin-Selberg $L$-factors} in \cite{BKL} and the {\it  formal exterior square $L$-functions} in \cite{MY}.

\par
Let $K_n={\rm GL}_n(\mathcal{O})$ be the maximal compact subgroup of $G_n$. For each non-negative integer $c$, we define the congruence subgroup $K_1(\mathfrak{p}^c)$ of $K_n$ by
\[
 K_1(\mathfrak{p}^c):=\left\{ g \in K_n\, \middle| \,g \equiv \left(\begin{array}{ccc} \multirow{2}{*}{$\ast$}\vspace{-1ex} &\vspace{-1ex} \ast  \\   & \vspace{-1ex} \vdots   \\ & \vspace{-1ex}   \ast \\ 
 0 \dotsm 0 & 1 \end{array}  \right) \; (\mathrm{mod}\; \mathfrak{p}^c)
  \right\}.
\]
We write $V^{K_1(\mathfrak{p}^c)}$ for $K_1(\mathfrak{p}^c)$-fixed vectors in $V$. One of the main results of \cite{JacquetCorrection,JP-SS81} is that there exists a non-negative $c$ such that $V^{K_1(\mathfrak{p}^c)} \neq \{ 0 \}$. We denote by $c(\pi)$ the smallest integer with this property. The nonnegative integer $c(\pi)$ is called the {\it conductor} of $\pi$ where we set $c(\pi)=0$ if $\pi$ is unramified. Then $V^{K_1(\mathfrak{p}^{c(\pi)})}$ is one-dimensional. There is a unique vector $v^{\circ}$ in the space $V^{K_1(\mathfrak{p}^{c(\pi)})}$ up to scalar multiplication, called the {\it essential vector} or the {\it newform}, with the associated essential Whittaker function $W_{\pi}^{\circ}:=(W_{\pi})_{v^{\circ}}$ satisfying the condition $W_{\pi}^{\circ}(I_n)=1$. If $\pi$ is unramifed, $W_{\pi}^{\circ}$ is what is said to be the normalized spherical Whitaker function (cf. \cite[\S 2]{BKL}). The explicit relation between the essential Whittaker function $W^{\circ}_{\pi}$ and the normalized spherical Whittaker function $W^{\circ}_{\pi_{ur}}$ has been independently unveiled by Matringe \cite[Corollary 3.2]{Ma13} and Miyauchi \cite[Theorem 4.1]{Miy}.

\begin{theorem}[Matringe and Miyauchi]
\label{NewV}
Let $\pi \in \mathcal{A}_F(n)$ be a ramified generic representation. Let $a=\mathrm{diag}(a_1,a_2,\dotsm,a_n)$ be a torus element and $a'$ a truncated element of $a$ given by $a'=\mathrm{diag}(a_1,a_2,\dotsm,a_r) \in A_r$. Then we have
\[
  W^{\circ}_{\pi} \begin{pmatrix} a & \\ & 1 \end{pmatrix}=W^{\circ}_{\pi_{ur}}(a') \nu^{\frac{n-r}{2}}(a') {\bf 1}_{\mathcal{O}}(a_r) \prod_{r < k < n} {\bf 1}_{\mathcal{O}^{\times}}(a_k) 
\]
\end{theorem}

Let $\mathcal{S}(F^n)$ be the space of locally constant and compactly supported functions $\Phi : F^n \rightarrow \mathbb{C}$. We define the test function $\Phi_{c} \in \mathcal{S}(F^n)$ by characteristic functions $ {\bf 1}_{\mathcal{O}^n}$ for $c=0$, and $ {\bf 1}_{({\mathfrak{p}^{c})}^{n-1} \times (1+\mathfrak{p}^{c})}$ for $c > 0$.
We highlight that the last entry of $ {\bf 1}_{\mathcal{O}^n}$ may not lie in the unit of the ring of integers for $c=0$. 

\par
For $c \geq 0$, depending on the choice of representations, we normalize the Haar measure on $N_n$, $A_n$, $P_n$ and $K_n$ so that the volumes of $N_n \cap  K_1(\mathfrak{p}^c)$, $A_n \cap  K_1(\mathfrak{p}^c)$, $P_n \cap  K_1(\mathfrak{p}^c)$, and $K_1(\mathfrak{p}^c)$ are all one respectively (cf. \cite[\S 3]{MY}). We embed $G_{n-1}$ into $G_n$ via the map $g \mapsto \begin{pmatrix} g & \\ & 1 \end{pmatrix}$. Regarded as subgroups of $G_{n-1}$, we normalize the Haar measure on $N_{n-1}$, $A_{n-1}$, and $K_{n-1}$ so that $\mathrm{vol}(N_{n-1} \cap K_{n-1})=\mathrm{vol}(A_{n-1} \cap K_{n-1})=\mathrm{vol}(K_{n-1})=1$ (cf. \cite[Proof of Theorem 6.1]{AM}).

\section{Rankin-Selberg $L$-factors}
\label{RS}

We review the definition of local $L$-functions attached to pairs, using the formulation of Jacquet, Piatetski-Shapiro, and Shalika \cite{JPSS3}. Let $\pi \in \mathcal{A}_F(n)$ and $\sigma \in  \mathcal{A}_F(m)$ be generic representations with associated Whittaker models $\mathcal{W}(\pi,\psi)$ and $\mathcal{W}(\sigma,\psi^{-1})$, respectively. 
Let $\{ e_i \;|\; 1 \leq i \leq n \}$ be the standard row basis of $F^n$. For each pair of Whittaker functions $W_{\pi} \in \mathcal{W}(\pi,\psi)$ and $W_{\sigma} \in \mathcal{W}(\sigma,\psi^{-1})$ and in the case $n=m$ each Schwartz-Bruhat function $\Phi \in \mathcal{S}(F^n)$, we associate the local Rankin-Selberg integrals
\[
 \Psi(s,W_{\pi},W_{\sigma})=\int_{N_m \backslash G_m} W_{\pi} \begin{pmatrix} g & \\ & I_{n-m} \end{pmatrix} W_{\sigma}(g)|\mathrm{det}(g)|^{s-\frac{n-m}{2}} dg
\]
in the case $n > m$, and in the case $n=m$
\[
  \Psi(s,W_{\pi},W_{\sigma},\Phi)=\int_{N_n \backslash G_n} W_{\pi}(g) W_{\sigma}(g) \Phi(e_ng) |\mathrm{det}(g)|^s dg.
  \]
  These integral converge absolutely for $\mathrm{Re}(s) \gg 0$. Let $\mathcal{I}(\pi \times \sigma)$ denote the complex linear span of the local integrals $\Psi(s,W_{\pi},W_{\sigma})$
  if $n > m$ and that of $\Psi(s,W_{\pi},W_{\sigma},\Phi)$ if $m=n$. The space $\mathcal{I}(\pi \times \sigma)$ is a $\mathbb{C}[q^{\pm s}]$-fractional ideal of $\mathbb{C}(q^{-s})$  containing the constant $1$. Since the ring $\mathbb{C}[q^{\pm s}]$ is a principal ideal domain and $1 \in \mathcal{I}(\pi \times \sigma)$, we can take a normalized generator of the form $P(q^{-s})^{-1}$ with $P(X) \in \mathbb{C}[X]$ having $P(0)=1$. The local Rankin-Selberg $L$-function attached to $\pi$ is defined by
  \[
   L(s,\pi \times \sigma)=\frac{1}{P(q^{-s})}.
  \]
   In particular, for a pair $(\pi,\sigma)$ of spherical representations, the local $L$-function $L(s,\pi \times \sigma)$ coincides with the formal local $L$-function $L(s,\pi_{ur}\times \sigma_{ur})$.
      
  \begin{proposition}\cite[Proposition 9.4]{JPSS3}
  \label{Langlands-RS}
   Let $\{ \alpha_i \}_{i=1}^r$ and $\{ \beta_j \}_{j=1}^p$ denote the Langlands parameter of $\pi$ and $\sigma$, respectively. Then we have
\[
  L(s,\pi_{ur} \times \sigma_{ur})=\prod_{i=1}^r \prod_{j=1}^p (1-\alpha_i(\varpi)\beta_j(\varpi)q^{-s})^{-1}.
\]
\end{proposition}

\subsection{The Rankin-Selberg period}

We construct a pair of Whittaker functions associated to newforms and the characteristic function such that the resulting Rankin-Selberg integral attains the formal tensor product $L$-factors.

\begin{theorem}
\label{RSNewVector}
Let $\pi, \sigma \in \mathcal{A}_F(n)$ be generic representations. We set $c=\max\{ c(\pi), c(\sigma) \}$. Then we have
\[
L(s,\pi_{ur} \times \sigma_{ur})=\Psi(s,W^{\circ}_{\pi},W^{\circ}_{\sigma}, \Phi_c).
\]
\end{theorem}

\begin{proof}
The computation is almost identical to \cite[Corollary 3.3]{Ma13}, and hence we only convey the key points. Based on the symmetry $L(s,\pi \times \sigma)=L(s,\sigma \times \pi)$, we may assume that $\sigma$ is ramified and $r \geq p$. Appealing to the Iwasawa decomposition, 
\[
\Psi(s,W^{\circ}_{\pi},W^{\circ}_{\sigma}, \Phi_c)=\int_{K_n}\int_{N_n \backslash P_n} W^{\circ}_{\pi}(pk) W^{\circ}_{\sigma}(pk)|\mathrm{det}(p)|^{s-1} \int_{F^{\times}} \Phi_c(e_nzpk) (\omega_{\pi}\omega_{\sigma})(z)|z|^{ns}d^{\times}zdpdk
\]
 As pointed out in \cite[Lemma 2.6]{MY}, $g \mapsto \Phi_c(e_ng)$ is the characteristic function on $P_nK_1(\mathfrak{p}^c)$ and this becomes
\[
\begin{split}
\Psi(s,W^{\circ}_{\pi},W^{\circ}_{\sigma}, \Phi_c)
 &=\int_{N_n \backslash P_n} W^{\circ}_{\pi}(p) W^{\circ}_{\sigma}(p)|\mathrm{det}(p)|^{s-1}  \int_{1+\mathfrak{p}^c}\Phi_c(e_nz) (\omega_{\pi} \omega_{\sigma})(z)|z|^{ns}\;d^{\times}zdp \\
 &=\int_{N_{n-1} \backslash G_{n-1}} W^{\circ}_{\pi}\begin{pmatrix} g& \\ & 1 \end{pmatrix} W^{\circ}_{\sigma}     \begin{pmatrix} g& \\ & 1 \end{pmatrix}     |\mathrm{det}(g)|^{s-1}  dp.
\end{split}
 \]
 We exploit Theorem \ref{NewV} aligned with the Iwasawa decomposition $G_{n-1}=N_{n-1}A_{n-1}K_{n-1}$ to obtain that
\[
\begin{aligned}
\Psi(s,W^{\circ}_{\pi},W^{\circ}_{\sigma}, \Phi_c)
  &=\int_{A_p} W^{\circ}_{\pi_{ur}}    \begin{pmatrix} a^{\prime} & \\ & I_{r-p} \end{pmatrix}   W^{\circ}_{\sigma_{ur}} (a^{\prime})  \delta_{B_{n-1}}^{-1}\begin{pmatrix} a^{\prime} & \\ & I_{n-p-1}\end{pmatrix} \nu^{(p-r)/2}(a^{\prime}) \nu^{n-p}(a^{\prime})   \\ 
  & \quad \times {\bf 1}_{\mathcal{O}}(a_p) |\mathrm{det}(a^{\prime})|^{s-1} da^{\prime}\\
  &=\int_{A_p} W^{\circ}_{\pi_{ur}} \begin{pmatrix} a^{\prime} & \\ & I_{r-p} \end{pmatrix} W^{\circ}_{\sigma_{ur}}(a^{\prime})  \delta_{B_{p}}^{-1}(a^{\prime}) \nu^{(p-r)/2}(a^{\prime})  \nu^{p-(n-1)}(a^{\prime}) \nu^{n-p}(a^{\prime}) \\
  & \quad \times {\bf 1}_{\mathcal{O}}(a_p)|\mathrm{det}(a^{\prime})|^{s-1}  da^{\prime}.
 \end{aligned}
\] 
  Once more by Iwasawa decomposition we get that
\[
\begin{split}
    \Psi(s,W^{\circ}_{\pi},W^{\circ}_{\sigma}, \Phi_c)
    &=\int_{A_p} W^{\circ}_{\pi_{ur}}\begin{pmatrix} a^{\prime} & \\ & I_{r-p} \end{pmatrix}   W^{\circ}_{\sigma_{ur}}(a^{\prime})  \delta_{B_{p}}^{-1}(a^{\prime})   {\bf 1}_{\mathcal{O}}(a_p)|\mathrm{det}(a^{\prime})|^{s-\frac{r-p}{2}}  da^{\prime}\\
    &=\begin{cases}
    \Psi(s,W^{\circ}_{\pi_{ur}},W^{\circ}_{\sigma_{ur}}, {\bf 1}_{\mathcal{O}^p}) &  \text{if $r=p$,} \\
     \Psi(s,W^{\circ}_{\pi_{ur}},W^{\circ}_{\sigma_{ur}} ) &  \text{if $r > p$,} \\
      \end{cases}
\end{split}
\]
  whence, applying the standard computation of local Rankin-Selberg $L$-functions for unramified representations from \cite[Proposition 2.3]{JaSh81} (cf. \cite[(3),(4)]{Ma13}), we have
  \[
   \Psi(s,W^{\circ}_{\pi},W^{\circ}_{\sigma}, \Phi_c)= L(s,\pi_{ur} \times \sigma_{ur})
      \]
      from which the desired result follows.
\end{proof}

Theorem \ref{RSNewVector} generalizes the result \cite[Theorem 2.1.1]{Kim} for the pair of the ramified representation $\pi$ and unramifed representation $\sigma$ in the viewpoint of the agreement $L(s,\pi \times \sigma)=L(s,\pi_{ur} \times \sigma_{ur})$ \cite[(5)]{Ma13} (cf. \cite[\S 7]{Venkatesh}). In what follows, we explain the link between $L(s,\pi_{ur} \times \sigma_{ur})$ and $L(s,\pi \times \sigma)$. 

\begin{corollary} 
\label{RS-divisibility}
Let $\pi, \sigma \in \mathcal{A}_F(n)$ be generic representations. Then we have
\[ 
L(s,\pi_{ur} \times \sigma_{ur})=P(q^{-s})L(s,\pi \times \sigma)
\]
for a polynomial $P(X) \in \mathbb{C}[X]$ satisfying $P(0)=1$.
\end{corollary}

\begin{proof}
The result is an immediate consequence of Theorem \ref{RSNewVector} that $L(s,\pi_{ur} \times \sigma_{ur})$ is an element of $\mathcal{I}(\pi \times \sigma)$.
\end{proof}

For the rest of this section, we navigate the bilinear form by the means of Rankin-Selbeg periods.

\begin{proposition}
Let $\pi, \sigma \in \mathcal{A}_F(n)$ be unitary generic representations. For any $W_{\pi} \in \mathcal{W}(\pi,\psi)$ and $W_{\sigma}  \in \mathcal{W}(\sigma,\psi^{-1})$, the integral
 \[
B(W_{\pi},W_{\sigma}):=\int_{N_n \backslash P_n} W_{\pi}(p) W_{\sigma}(p) dp
 \]
 is absolutely convergent and defines a $P_n$-invariant bilinear form on $\mathcal{W}(\pi,\psi) \times \mathcal{W}(\sigma,\psi^{-1})$.
\end{proposition}

\begin{proof}
A parallel discussion can be found in \cite[Lemma 3.1]{Ma10}. Let $K_n^{\circ}$ be a compact open subgroup of $K_n$ such that each $W_{\pi}$ and $W_{\sigma}$ is invariant under 
$K_n^{\circ}$. We choose $\Phi^{\circ}$ to be a characteristic function of $e_n K_n^{\circ}$. Then the integral $\Psi(s,W_{\pi},W_{\sigma}, \Phi^{\circ})$ reduces a positive multiple of
\begin{equation}
\label{mirabolic}
 \int_{N_n \backslash P_n} W_{\pi}(p) W_{\sigma}(p) |\mathrm{det}(p)|^{s-1} \,dp.
\end{equation}
According to \cite[Proposition 3.17]{JaSh81}, the integral $\Psi(s,W_{\pi},W_{\sigma},\Phi^{\circ})$ converges absolutely in the half plane $\mathrm{Re}(s) \geq 1$. This confirms that the integral \eqref{mirabolic} is holomorphic at $s=1$.
\end{proof}

We embark on proving the non-vanishing of the $P_n$-bilinear form $B(W_{\pi},W_{\sigma})$. This is similar to the canonical inner product considered in \cite[Appendix A.1]{FLO}.

\begin{theorem}
\label{RS-priods}
Let $\pi, \sigma \in \mathcal{A}_F(n)$ be unitary generic representations. Then we have
\[
  \int_{N_n \backslash P_n} W^{\circ}_{\pi}(p) W^{\circ}_{\sigma}(p) dp=
  \begin{cases}
  L(1,\pi_{ur} \times \sigma_{ur})& \text{if either $\pi$ or $\sigma$ is ramified,} \\
  \dfrac{L(1,\pi \times \sigma)}{L(n,\omega_{\pi}\omega_{\sigma})}& \text{if both $\pi$ and $\sigma$ are unramified.} 
  \end{cases}
\]
\end{theorem}

\begin{proof}
The unramified case is proceeded along the line of the proof of \cite[Theorem 6.2]{AM}. We assume that $\sigma$ is ramified and $r \geq p$. Invoking the Iwasawa decomposition $G_{n-1}=N_{n-1}A_{n-1}K_{n-1}$, we have
\[
\begin{split}
   B(W^{\circ}_{\pi},W^{\circ}_{\sigma})
   &=\int_{N_{n-1} \backslash G_{n-1}} W^{\circ}_{\pi} \begin{pmatrix} g& \\ & 1 \end{pmatrix} W^{\circ}_{\sigma}  \begin{pmatrix} g& \\ & 1 \end{pmatrix} \;dg \\
   &=\int_{A_{n-1}} \int_{K_{n-1}}  W^{\circ}_{\pi} \begin{pmatrix} ak& \\ & 1 \end{pmatrix} W^{\circ}_{\sigma}  \begin{pmatrix} ak& \\ & 1 \end{pmatrix} \delta_{B_{n-1}}^{-1}(a) \;dk da \\
   &=\int_{A_{n-1}} W^{\circ}_{\pi} \begin{pmatrix} a& \\ & 1 \end{pmatrix} W^{\circ}_{\sigma}  \begin{pmatrix} a& \\ & 1 \end{pmatrix} \delta_{B_{n-1}}^{-1}(a) da. \\
\end{split}
\]
Chasing the steps of the proof of Theorem \ref{RSNewVector}, we arrive at
 \[
B(W^{\circ}_{\pi},W^{\circ}_{\sigma})
 =\int_{A_p} W^{\circ}_{\pi_{ur}}\begin{pmatrix} a^{\prime} & \\ & I_{r-p} \end{pmatrix} W^{\circ}_{\sigma_{ur}}(a^{\prime})  \delta_{B_{p}}^{-1}(a^{\prime})  {\bf 1}_{\mathcal{O}}(a_p) |\mathrm{det}(a^{\prime})|^{1-\frac{r-p}{2}} da^{\prime},
\]
which equates to $ \Psi(1,W^{\circ}_{\pi_{ur}},W^{\circ}_{\sigma_{ur}}, {\bf 1}_{\mathcal{O}^p})$ if $r=p$, and $ \Psi(1,W^{\circ}_{\pi_{ur}},W^{\circ}_{\sigma_{ur}} ) $ if $r > p$.
The integral is simply a reformulation of \cite[Proposition 2.3]{JaSh81} (see \cite[(3),(4)]{Ma13}).
 \end{proof}

\subsection{The self-dual representation} 

A representation $\pi$ of $G_n$ is called {\it self-dual}, if $\pi^{\vee} \simeq \pi$.
We will see later in \S \ref{JSS2n}, \S \ref{BJH2n}, and \S \ref{BG} that either $(S_{2n},\Theta)$-distinguished, $H_{2n}$-distinguished, or $\theta$-distinguished representation is self-dual.

\begin{proposition}
\label{Bernstein-s=1}
Let $\pi \in \mathcal{A}_F(n)$ be a generic representation of $G_n$. Then $L(s,\pi \times \pi^{\vee})$ is holomorphic at $s=1$. In particular the bilinear form $B(W_{\pi},W_{\pi^{\vee}})$ on $\mathcal{W}(\pi,\psi) \times \mathcal{W}(\pi^{\vee},\psi^{-1})$ is well-defined.
 \end{proposition}

\begin{proof}
For $\mathrm{Re}(s) \gg 0$, we may decompose our integral as
\[
 \Psi(s,W_{\pi},W_{\pi^{\vee}},\Phi)=\int_{K_n} \int_{N_n \backslash P_n} W_{\pi}(pk) W_{\pi^{\vee}}(pk) |\mathrm{det}(p)|^{s-1} \int_{F^{\times}} \Phi(e_nzk) |z|^{ns} d^{\times} z dp dk.
\]
We take $K_n^{\circ} \subseteq K_n$ a compact open subgroup which stabilizes $W_{\pi}$, $W_{\pi^{\vee}}$, and $\Phi$. We write $K_n=\underset{{i: \,{\rm finite}}}{\cup} k_i K_n^{\circ}$. 
The integral $\Psi(s,W_{\pi},W_{\pi^{\vee}},\Phi)$ can be decomposed as a finite sum of the form
\[
 \Psi(s,W_{\pi},W_{\pi^{\vee}},\Phi)=\sum_{i} v \int_{F^{\times}} \Phi(e_nzk_i) |z|^{ns} d^{\times} z  \int_{N_n \backslash P_n}  (\pi(k_i)W_{\pi})(p) (\pi^{\vee}(k_i)W_{\pi^{\vee}})(p)  |\mathrm{det}(p)|^{s-1}dp
\]
with $v > 0$ a volume term. Bernstein's Theorem \citelist{\cite{Bern}*{Theorem 6.4} \cite{FLO}*{Appendix A}} guarantees that 
\[
\int_{N_n \backslash P_n}  (\pi(k_i)W_{\pi})(p) (\pi^{\vee}(k_i)W_{\pi^{\vee}})(p)  |\mathrm{det}(p)|^{s-1}dp
\]
are holomorphic at $s=1$ and we note that Tate integrals
\[
\int_{F^{\times}} \Phi(e_nzk_i) |z|^{ns} d^{\times} z
\]
are absolutely convergent for $\mathrm{Re}(s) > 0$. This amounts to saying that the integral $\Psi(s,W_{\pi}, W_{\pi^{\vee}},\Phi)$ is regular at $s=1$.
\end{proof}

We say that $\pi \otimes \sigma$ is $G_n$-{\it distinguished} if ${\rm Hom}_{G_n}(\pi \otimes \sigma,{\textbf 1}_{G_n}) \neq 0$. It is worthwhile to point out that our condition of the distinction for pairs $(\pi,\sigma)$ is equivalent to the condition $\sigma \simeq \pi^{\vee}$ in \cite[\S 6.2]{MO} (cf. \cite[Remark 7.4]{Yamana15}). 
In particular for $\pi \simeq \sigma$, $\pi$ is self-dual if $\pi \otimes \pi$ is $G_n$-distinguished.
Additionally any $G_n$-distinguished representation $\pi \otimes \sigma$ always satisfies $\omega_{\pi}\omega_{\sigma}={\bf 1}_{F^{\times}}$. The unitarity hypothesis is redundant once we demand the distinction criterion on the pairs $(\pi,\sigma)$.


\begin{theorem}
\label{RS-distinction}
Let $\pi, \sigma \in \mathcal{A}_F(n)$ be a generic representation such that $\pi \otimes \sigma$ is $G_n$-distinguished. Then we have
\[
  \int_{N_n \backslash P_n} W^{\circ}_{\pi}(p) W^{\circ}_{\pi^{\vee}}(p) dp=
  \begin{cases}
  L(1,\pi_{ur} \times \pi^{\vee}_{ur})&  \text{if $\pi$ is ramified,} \\
  \dfrac{L(1,\pi \times \pi^{\vee})}{L(n, {\bf 1}_{F^{\times}}   )}& \text{otherwise.} 
  \end{cases}
\]
\end{theorem}

\begin{proof}
The assumption ${\rm Hom}_{G_n}(\pi \otimes \sigma,{\bf 1}_{G_m}) \neq 0$ implies that $\pi^{\vee} \cong \sigma$. The result is a special case of Theorem \ref{RS-priods} once we identify $\sigma$ with $\pi^{\vee}$ and then utilize the reguliarty of $L(s,\pi \times \pi^{\vee})$ at $s=1$, as in Proposition \ref{Bernstein-s=1}.
\end{proof}

Although the following result has been known to experts \citelist{\cite{Bern} \cite{FLO}}, we cannot locate in the form that we need in the paper. Hence we include the statement and its proof.

\begin{proposition}
\label{RS-multiplicityone}
Let $\pi, \sigma \in \mathcal{A}_F(n)$ be unitary representations. Then
\[
  {\rm dim}_{\mathbb{C}} {\rm Hom}_{P_n}(\pi \otimes \sigma, {\bf 1}_{P_{n}}) \leq 1.
\]
The equality holds when $\pi$ and $\sigma$ are generic. In particular, if $\pi \otimes \sigma$ is $G_n$-distinguished, then $(W_{\pi},W_{\sigma}) \mapsto B(W_{\pi},W_{\sigma})$ gives a unique non-trivial $G_n$-invariant bilinear form belonging to the space ${\rm Hom}_{G_n}(\pi \otimes \sigma,{\bf 1}_{G_n})$.  
\end{proposition}

\begin{proof}
The proof is a variation on those of \citelist{\cite{Ma14}*{Proposition 2.3}  \cite{MA15}*{Corollary 4.2} \cite{Yamana}*{Theorem 2.14} } and falls out of that of \cite[Proposition 2.10]{JPSS3}. 
Both representations $\pi|_{P_n}$ and $(\sigma|_{P_n})^{\vee}$ carry Bernstein-Zelevinsky filtrations of $P_n$-submodules
\[
 0 \subseteq \pi_n \subseteq \pi_{n-1} \subseteq \dotsm \subseteq \pi_1:=\pi|_{P_n} \quad \text{and} \quad
  0 \subseteq \tau_n \subseteq \tau_{n-1} \subseteq \dotsm \subseteq \tau_1:=(\sigma|_{P_n})^{\vee} 
\]
such that $\pi_k / \pi_{k+1} = (\Phi^+)^{k-1}\Psi^+(\pi_1^{(k)})$ and $\tau_k / \tau_{k+1} = (\Phi^+)^{k-1}\Psi^+(\tau_1^{(k)})$. On the one hand, $\Psi^+$ is normalized inflation and   
$\Phi^+$ is normalized compactly supported induction. On the other hand, $\Psi^-$ is the normalized Jacquet functor and $\Phi^-$ is  the normalized $\psi$-twisted Jacquet functor. For the rigorous definition of the four functors $\Phi^+$, $\Psi^+$, $\Phi^-$, and $\Psi^-$, the reader should consult \cite[\S 3.2]{BeZe}. $\pi_1^{(k)}$ is so-called {\it Bernstein-Zelevinsky $k^{th}$-derivatives}, and for our purpose it is convenient to introduce the shifted derivatives $\pi_1^{[k]}:=\pi_1^{(k)}\otimes\nu^{1/2}$. We conclude from \cite[Proposition 1.4]{JPSS3} that
\[
  {\rm Hom}_{P_n}((\Phi^+)^{i-1}\Psi^+(\pi_1^{(i)}), (\Phi^+)^{j-1}\Psi^+(\tau_1^{(j)}))=0
\]
except when $i=j$. Hence if ${\rm Hom}_{P_n}(\pi \otimes \sigma,\mathbb{C})$ is non-zero, it is clear from \cite[Proposition 3.2]{BeZe} that the space must 
induce a non-zero space
\[
  {\rm Hom}_{P_n}((\Phi^+)^{k-1}\Psi^+(\pi_1^{(k)}), (\Phi^+)^{k-1}\Psi^+(\tau_1^{(k)})) \simeq 
   {\rm Hom}_{G_{n-k}}(\pi_1^{(k)},\tau_1^{(k)})
\]
for some $1 \leq k \leq n$. We put $\sigma_1:=\sigma|_{P_n}$. Exchanging the order of functors $\Psi^-$ and $\Phi^-$, and the duality \cite[Proposition 3.4]{BeZe} gives rise to 
\[
\tau_1^{(k)}=\Psi^-(\Phi^-)^{k-1}((\sigma|_{P_n})^{\vee} ) \simeq \nu^{-1} \otimes (\Psi^-(\Phi^-)^{k-1}(\sigma|_{P_n} ))^{\vee}
\]
from which we deduce the isomorphism;
\[
  {\rm Hom}_{G_{n-k}}(\pi_1^{(k)},\tau_1^{(k)}) \simeq {\rm Hom}_{G_{n-k}}(\pi_1^{(k)},\nu^{-1} \otimes (\sigma_1^{(k)})^{\vee})
  \simeq  {\rm Hom}_{G_{n-k}}(\pi_1^{[k]} \otimes \sigma_1^{[k]} ,{\bf 1}_{G_{n-k}}).
\]
Without loss of generality, we may assume that $\pi_1^{(h)} \neq 0$, $\sigma_1^{(h)} \neq 0$ and $\pi_1^{(k)}=\sigma_1^{(k)}=0$ for all $k > h$. The Bernstein's criterion in \cite[\S 7.4]{Bern} states that $\pi_1^{[h]}$ and $\sigma_1^{[h]}$ are irreducible and unitary, and central characters of the irreducible subquotients of $\pi_1^{[k]}$ and $\sigma_1^{[k]}$ have real parts greater than zero for all $0 < k < h$. In this way, we find that
\[
   {\rm dim}_{\mathbb{C}} {\rm Hom}_{P_n}(\pi \otimes \sigma,{\bf 1}_{P_n}) \leq 
   {\rm dim}_{\mathbb{C}}  {\rm Hom}_{G_{n-h}} (\pi_1^{[h]} \otimes \sigma_1^{[h]} ,{\bf 1}_{G_{n-h}}).
\]
Now our proof is completed by induction on the rank $n-h$ of $G_{n-h}$.

\par
To prove the uniqueness of the bilinear form, we substitute ${\rm Hom}_{G_n}(\pi \otimes \sigma,{\bf 1}_{G_n})$ by ${\rm Hom}_{G_n}(\pi \otimes \pi^{\vee},{\bf 1}_{G_n})$ as in Theorem \ref{RS-distinction}. The injection
$\mathrm{Hom}_{G_n}(\pi \otimes \pi^{\vee},{\bf 1}_{G_n}) \hookrightarrow \mathrm{Hom}_{P_n}(\pi \otimes\pi^{\vee},{\bf 1}_{P_n}) $
is interpreted as the isomorphism in view of
$
  0 \neq \mathrm{Hom}_{G_n}(\pi \otimes \sigma, {\bf 1}_{G_n})  \simeq \mathrm{Hom}_{G_n}(\pi \otimes \pi^{\vee},{\bf 1}_{G_n}) \simeq \mathrm{Hom}_{P_n}(\pi \otimes \pi^{\vee},{\bf 1}_{P_n}). $
\end{proof}

We do not claim the originality of the second assertion in Proposition \ref{RS-multiplicityone} as it is a special case of Bernstein's Theorem \cite[Theorem A]{Bern}.


\section{Asai $L$-factors}
\label{Asai}
We overview the theory of Asai $L$-factors in the appendix to \cite{Fil93}. Let $E$ be a quadratic extension of $F$. Thanks to \cite[Lemma 4.2.]{AM}, we take $\psi_E$ to be a non-trivial unramifed additive character of $E$ that is trivial on $F$.
Let $\pi \in \mathcal{A}_E(n)$ be a generic representation with an associated Whittaker model $\mathcal{W}(\pi,\psi_E)$. For each Whittaker function $W_{\pi} \in \mathcal{W}(\pi,\psi_E)$
and each Schwartz-Bruhat function $\Phi \in S(F^n)$, we define local Flicker integral by
\[
 I(s,W_{\pi},\Phi)=\int_{N_n \backslash {\rm GL}_n(F)} W_{\pi}(g)\Phi(e_ng)|\mathrm{det}(g)|^s dg
\]
which is absolutely convergent when the real part of $s$ is sufficiently large enough. Each $I(s,W_{\pi},\Phi)$ is a rational function of $q^{-s}$ and hence extends meromorphically to all of $\mathbb{C}$. These integrals $I(s,W_{\pi},\Phi)$ span a fractional ideal $\mathcal{I}(\pi,As)$ of $\mathbb{C}[q^{\pm s}]$ generated by a normalized generator of the form $P(q^{-s})^{-1}$ where the polynomial $P(X) \in \mathbb{C}[X]$ satisfies $P(0)=1$. The local Asai $L$-function attached to $\pi$ is defined by such a unique normalized generator;
  \[
   L(s,\pi, As)=\frac{1}{P(q^{-s})}.
  \]
   In particular, for a spherical representation $\pi$, the local $L$-function $L(s,\pi, As)$ is equal to the formal local $L$-function $L(s,\pi_{ur},As)$.
\begin{proposition}\cite[Theorem 4.26]{Ma09}
Let $\{ \alpha_i \}_{i=1}^r$ denote the Langlands parameter of $\pi$. Then we have
\[
L(s,\pi_{ur},As)=\prod_{k=1}^r (1-\alpha_k(\varpi_F)q_F^{-s})^{-1} \prod_{1 \leq i < j \leq r} (1-\alpha_i(\varpi_E)\alpha_j(\varpi_E)q_E^{-s})^{-1}
\]
\end{proposition}

\subsection{The Flicker-Rallis period}

We are now going to produce an essential Whittaker function and the characteristic function so that
the resulting Flicker integral accomplishes the formal Asai $L$-factors.

\begin{theorem}
\label{test-Asai}
Let $\pi \in \mathcal{A}_E(n)$ be a generic representation. We set $c=c(\pi)$. Then we have
\[
L(s,\pi_{ur},As)=I(s,W^{\circ}_{\pi},\Phi_c).
\]
\end{theorem}

\begin{proof} 
Having in hand \cite[Lemma 2.6]{MY} that $g \mapsto \Phi_c(e_ng)$ is the characteristic function on $P_nK_1(\mathfrak{p}^c)$, the partial Iwasawa decomposition $G_n=N_nP_nK_n$ implies that
\begin{align*}
I(s,W^{\circ}_{\pi},\Phi_c)&=\int_{K_1(\mathfrak{p}^c)}\int_{N_n \backslash P_n} W^{\circ}_{\pi}(pk)  |\mathrm{det}(p)|^{s-1} \int_{F^{\times}} \Phi_c(e_nzpk) \omega_{\pi}(z)|z|^{ns} d^{\times}zdpdk\\
&=\int_{N_n \backslash P_n} W^{\circ}_{\pi}(p)  |\mathrm{det}(p)|^{s-1} \int_{1+\mathfrak{p}^c}
 \Phi_c(e_nz) \omega_{\pi}(z)|z|^{ns} d^{\times}zdp\\
 &=\int_{N_n \backslash P_n} W^{\circ}_{\pi}(p)  |\mathrm{det}(p)|^{s-1}dp.
\end{align*}
Repeating the procedure of the proof of \cite[Theorem 6.1]{AM}, the integral becomes
\[
I(s,W^{\circ}_{\pi},\Phi_c)=\int_{A_r} W^{\circ}_{\pi_{ur}}(a)\delta_{B_{r}}^{-1}(a)  {\bf 1}_{\mathcal{O}}(a_r) |\mathrm{det}(a)|^s da
=I(s,W^{\circ}_{\pi_{ur}},{\bf 1}_{\mathcal{O}^r}).
\]
Our expected result is an immediate consequence of the computation by \cite[Proposition 3]{Fil88} for $E \slash F$ the unramified extension and by  \cite[Proposition 9.5]{AKMSS} for $E \slash F$ the ramified extension.
\end{proof}

In general, $L(s,\pi_{ur}, As)$ is not the same as $L(s,\pi, As)$.

\begin{corollary} Let $\pi \in \mathcal{A}_E(n)$ be a generic representation. Then we have
\[ 
L(s,\pi_{ur}, As)=P(q^{-s})L(s,\pi, As)
\]
for a polynomial $P(X) \in \mathbb{C}[X]$ satisfying $P(0)=1$.
\end{corollary}

\begin{proof}
The result immediately follows from Theorem \ref{test-Asai} that $L(s,\pi_{ur}, As)$ is an element of $\mathcal{I}(\pi,As)$.
\end{proof}

We now wish to begin summarizing the main result on the $P_n(F)$-invariant form and the Flicker-Rallis integral periods from \cite[Theorem 1.1, Theorem 6.2]{AM}.

\begin{theorem}[Anandavardhanan and Matringe] 
\label{Asai-FJ}
Let $\pi \in \mathcal{A}_E(n)$ be a unitary generic representation. 
\label{unitary-flicker}
\[
 \int_{N_n \backslash P_n} W^{\circ}_{\pi}(p) dp
 =
  \begin{cases}
  L(1,\pi_{ur},As )& \text{if $\pi$ is ramified,} \\
  \dfrac{L(1,\pi, As )}{L(n,\omega_{\pi}|_{F^{\times}})}& \text{if $\pi$ is unramified.} 
  \end{cases}
\]
\end{theorem}

\subsection{The Galois self-conjugate dual representation} 
A representation $\pi$ of ${\rm GL}_n(E)$ is called $G_n$-{\it distinguished} if $\mathrm{Hom}_{G_n}(\pi|_{G_n},{\bf 1}_{G_n}) \neq 0$. We denote by $\sigma_{E \slash F} : x \mapsto \overline{x}, x \in E$ the non-trivial associated Galois action. The conjugation $\sigma_{E \slash F}$ extends naturally to an automorphism of ${\rm GL}_n(E)$ which we also denote by $\sigma_{E \slash F}$. Then we denote by $\pi^{\sigma_{E \slash F}}$ the representation $\pi^{\sigma_{E \slash F}}(g)=\pi(\sigma_{E \slash F}(g))$ for $g \in {\rm GL}_n(E)$ and $\pi \in \mathcal{A}_E(n)$.
We say that a representation $\pi$ of ${\rm GL}_n(E)$ is the {\it Galois self-conjugate dual} representation if $\pi^{\sigma_{E \slash F}} \simeq \pi^{\vee}$. In \cite[Proposition 12]{Fli91}, Flicker proved that if $\pi \in \mathcal{A}_E(n)$ is $G_n$-distinguished, then it is Galois self-conjugate dual. 
A kind of converse can be established for discrete series representations by the work of Anandavardhanan and  Rajan \cite[Theorem 4]{AR05}, and Kable \cite[Theorem 6]{Kal04}.
We refer the reader to \cite[Theorem 5.2]{Mat11} for the complete classification of $G_n$-distinguished generic representations, in terms of inducing discrete series representations.

\begin{theorem}[Anandavardhanan and Matringe] Let $\pi \in \mathcal{A}_E(n)$ be a generic representation which is distinguished with respect to $G_n$. Then we have
\[
 \Lambda^{FR}(W^{\circ}_{\pi}):
= \int_{N_n \backslash P_n} W^{\circ}_{\pi}(p) dp
  =
  \begin{cases}
  L(1,\pi_{ur},As )& \text{if $\pi$ is ramified,} \\
  \dfrac{L(1,\pi, As )}{L(n, {\bf 1}_{F^{\times}})}& \text{if $\pi$ is unramified.} 
  \end{cases}
\]
\end{theorem}

As illustrated in \S \ref{Intro} Introduction, the non-trivial linear functional $\Lambda^{FR}$ yields non-zero $P_n$-invariant forms on $\mathcal{W}(\pi,\psi_E)$, which in turn can be uniquely extended to $G_n$-invariant forms on  $\mathcal{W}(\pi,\psi_E)$. 

\begin{proposition} \citelist{\cite{AKT}*{Theorem 1.1} \cite{Ma14}*{Proposition 2.3}}
Let $\pi \in \mathcal{A}_E(n)$ be a unitary representation. Then 
\[
{\rm dim}_{\mathbb{C}} {\rm Hom}_{P_n}(\pi|_{P_n},{\bf 1}_{P_n}) \leq 1.
\]
The equality holds when $\pi$ is generic. In particular, if $\pi$ is $G_n$-distinguished, then $W_{\pi} \mapsto \Lambda^{FR}(W_{\pi})$ gives a unique non-trivial $G_n$-invariant bilinear form belonging to ${\rm Hom}_{G_n}(\pi|_{G_n},{\bf 1}_{G_n})$. 
\end{proposition}

Once more, the unitarity assumption on the second statement can be dropped, appealing to Ok's Theorem \cite[Proposition 2.3]{Ma14} coupled with \cite[Proposition 11]{Fli91}.

\section{The Exterior Square $L$-factors}

\subsection{The Jacquet-Shalika period}
\label{JSperiod}

We give a short discussion on the theory of local exterior square $L$-functions due to Jacquet and Shalika \citelist{\cite{JaSh88} \cite{Jo20}*{\S 2.2,\,\S 2.4}}.
Let $\sigma_m$ be the permutation matrix given by
\[
 \sigma_{2n}=\begin{pmatrix} 1 & 2 & \dotsm & n & | & n+1 & n+2 & \dotsm & 2n \\ 
                                                1 & 3 & \dotsm & 2n-1 & | &  2 & 4 & \dotsm &2n \\ 
                                                \end{pmatrix}
\]
when $m=2n$ is even, and by
\[
 \sigma_{2n+1}=\begin{pmatrix} 1 & 2 & \dotsm & n & | & n+1 & n+2 & \dotsm & 2n & 2n+1 \\ 
                                                1 & 3 & \dotsm & 2n-1 & | &  2 & 4 & \dotsm &2n & 2n+1 \\ 
                                                \end{pmatrix}
\]
when $m=2n+1$ is odd. We let $\mathcal{M}_n$ be the $n \times n$ matrices, $\mathcal{N}_n$ the subspace of upper triangular matrices of $\mathcal{M}_n$. For a Whittaker functions $W_{\pi} \in \mathcal{W}(\pi,\psi)$ and in the even case $m=2n$ each Schwartz-Bruhat function $\Phi \in \mathcal{S}(F^n)$, we define local Jacquet-Shalika integrals
\begin{equation}
\label{JS-odd}
  J(s,W_{\pi}):=\int_{N_n \backslash G_n}  \int_{\mathcal{N}_n \backslash \mathcal{M}_n} W_{\pi} \left( \sigma_{2n+1} \begin{pmatrix} I_n &X& \\ &I_n& \\ &&1 \end{pmatrix} \begin{pmatrix} g && \\ &g& \\ &&1 \end{pmatrix} \right)\psi^{-1}({\rm Tr}(X)) |\mathrm{det}(g)|^{s-1}  dX dg
\end{equation}
in the odd case $m=2n+1$ and in the even case $m=2n$
\begin{equation}
\label{JS-even}
J(s,W_{\pi},\Phi):= \int_{N_n \backslash G_n} \int_{\mathcal{N}_n \backslash \mathcal{M}_n} W_{\pi} \left( \sigma_{2n} \begin{pmatrix} I_n & X \\ & I_n  \end{pmatrix} \begin{pmatrix} g & \\ & g \end{pmatrix} \right)  \psi^{-1}({\rm Tr}(X)) \Phi(e_ng) |\mathrm{det}(g)|^s dX dg
\end{equation}
all integrals being convergent for $\mathrm{Re}(s) \gg 0$. Let $\mathcal{J}(\pi,\wedge^2)$ denote the complex linear span of the local Jacquet-Shalika integrals $J(s,W_{\pi})$
  if $m=2n+1$ is odd and that of $J(s,W_{\pi},\Phi)$ if $m=2n$ is even. The space $\mathcal{J}(\pi,\wedge^2)$ is a $\mathbb{C}[q^{\pm s}]$-fractional ideal of $\mathbb{C}(q^{-s})$  containing the constant $1$. This ideal $\mathcal{J}(\pi,\wedge^2)$ is principal, and has a unique generator of the form $P(q^{-s})^{-1}$, where $P(X)$ is a polynomial in $\mathbb{C}[X]$
with $P(0)=1$. The local exterior square $L$-function attached to $\pi$ is defined by the unique normalized generator;
  \[
   L(s,\pi, \wedge^2)=P(q^{-s})^{-1}.
  \]
   In particular, for a spherical representation $\pi$, the local $L$-function $L(s,\pi, \wedge^2)$ agrees with the formal local $L$-function $L(s,\pi_{ur},\wedge^2)$.

 \begin{proposition}\cite[Theorem 5.7]{Jo20}
 \label{formal-exterior}
  Let $\{ \alpha_i \}_{i=1}^r$ denote the Langlands parameter of $\pi$. Then we have
\[
  L(s,\pi_{ur},\wedge^2)= \prod_{1 \leq i < j \leq r} (1-\alpha_i(\varpi)\alpha_j(\varpi)q^{-s})^{-1}.
\]
\end{proposition}
We take the Haar measure on the quotient space $\mathcal{N}_n \backslash \mathcal{M}_n$ so that the volume of $\mathcal{N}_n \backslash (\mathcal{N}_n +\mathcal{M}_n(\mathcal{O}))$ is one. The test vector problem for exterior square $L$-factors has been carried out in Theorem 3.1 and Theorem 4.1 of \cite{MY}.

\begin{theorem}[Miyauchi and Yamauchi]
\label{JSEQ}
Let $\pi \in \mathcal{A}_F(m)$ be a generic representation. We set $c=c(\pi)$.
\begin{enumerate}[label=$(\mathrm{\alph*})$]
\item\label{JSEQ-1} When $m=2n$ is even, we have
$
 L(s,\pi_{ur}, \wedge^2)=
  J(s,W^{\circ}_{\pi},\Phi_c).
$
\item\label{JSEQ-2} When $m=2n+1$ is odd, we obtain
$
L(s,\pi_{ur}, \wedge^2)=J(s,W^{\circ}_{\pi}).
$
\end{enumerate}
\end{theorem}

In general, $L(s,\pi_{ur},\wedge^2)$ does not match with $L(s,\pi,\wedge^2)$.

\begin{corollary} Let $\pi \in \mathcal{A}_F(m)$ be a generic representation. Then we have
\[ 
L(s,\pi_{ur},\wedge^2)=P(q^{-s})L(s,\pi,\wedge^2)
\]
for a polynomial $P(X) \in \mathbb{C}[X]$ satisfying $P(0)=1$.
\end{corollary}

\begin{proof}
We only need to observe from Theorem \ref{JSEQ} that $L(s,\pi_{ur},\wedge^2)$ belongs to the $\mathbb{C}[q^{\pm s}]$-fractional ideal $\mathcal{J}(\pi,\wedge^2)$ of $\mathbb{C}(q^{-s})$.
\end{proof}

We concern with the following integrals, of which the convergence will be elaborated shortly after;
\[
 \Lambda^{JS}(W_{\pi}):= \int_{N_n \backslash P_n}  \int_{\mathcal{N}_n \backslash \mathcal{M}_n} W_{\pi} \left( \sigma_{2n} \begin{pmatrix} I_n & X \\ & I_n  \end{pmatrix} \begin{pmatrix} p & \\ & p \end{pmatrix} \right) \psi^{-1}({\rm Tr}(X)) dX dp 
\]
if $m=2n$ is even, and
\[
\Lambda^{JS}(W_{\pi}):=\int_{N_n \backslash G_n}  \int_{\mathcal{N}_n \backslash \mathcal{M}_n} W_{\pi} \left( \sigma_{2n+1} \begin{pmatrix} I_n &X& \\ &I_n& \\ &&1 \end{pmatrix} \begin{pmatrix} g && \\ &g& \\ &&1 \end{pmatrix} \right)\psi^{-1}({\rm Tr}(X)) dX dg 
 \]
if $m=2n+1$ is odd. We define the Shalika subgroup $S_{2n}$ of $G_{2n}$ by
\[
  S_{2n}=\left\{ \begin{pmatrix} I_n & Z \\ & I_n \end{pmatrix} \begin{pmatrix} g &  \\ & g \end{pmatrix}\; \middle|\;  Z \in \mathcal{M}_n, g \in G_{n} \right\}.
\]
Let $\Theta$ be a Shalika character of $S_{2n}$ given by
\[
 \Theta \left( \begin{pmatrix} I_n & Z \\ & I_n \end{pmatrix} \begin{pmatrix} g &  \\ & g \end{pmatrix} \right)=\psi(\mathrm{Tr}(Z)).
\]

\begin{proposition}
\label{Shalika-convergence}
 Let $\pi \in \mathcal{A}_F(m)$ be a unitary generic representation. For any $W_{\pi} \in \mathcal{W}(\pi,\psi)$, the integrals $\Lambda^{JS}(W_{\pi})$
converge absolutely. In particular, for $m=2n$ even, $W_{\pi} \mapsto \Lambda^{JS}(W_{\pi})$ defines a $(P_{2n} \cap S_{2n},\Theta)$-invariant linear functional on $\mathcal{W}(\pi,\psi)$.
\end{proposition}

\begin{proof}
The odd case can be extracted from \cite[\S 9.3 Proposition 3]{JaSh88}. Pertaining to even cases, we choose $K_n^{\circ}$ a compact open subgroup of $K_n$ such that $W_{\pi}$ is invariant under $\begin{pmatrix} K^{\circ}_n & \\ & K^{\circ}_n \end{pmatrix}$. We take $\Phi^{\circ}$ to be a characteristic function $e_nK_n^{\circ}$. By the partial Iwasawa decomposition $G_n=Z_nN_nK_n$, the integral $J(s,W_{\pi},\Phi^{\circ})$ is written as
\[
\begin{split}
J(s,W_{\pi},\Phi^{\circ})
=&\int_{K_n^{\circ}} \int_{N_n \backslash P_n} \int_{\mathcal{N}_n \backslash \mathcal{M}_n} W_{\pi} \left( \sigma_{2n} \begin{pmatrix} I_n & X \\ & I_n  \end{pmatrix} \begin{pmatrix} pk & \\ & pk \end{pmatrix} \right) \psi^{-1}({\rm Tr}(X))  |\mathrm{det}(p)|^{s-1} \\
&\times \int_{F^{\times}}  \Phi^{\circ}(e_nzk) |z|^{ns} d^{\times} z dX dp dk.
\end{split}
\]
By the right $K_n^{\circ}$ invariance of integrands, the integral $J(s,W_{\pi},\Phi^{\circ})$ is converted to
\[
 J(s,W_{\pi},\Phi^{\circ})=v \int_{N_n \backslash P_n}  \int_{\mathcal{N}_n \backslash \mathcal{M}_n} W_{\pi} \left( \sigma_{2n} \begin{pmatrix} I_n & X \\ & I_n  \end{pmatrix} \begin{pmatrix} p & \\ & p \end{pmatrix} \right) \psi^{-1}({\rm Tr}(X))  |\mathrm{det}(p)|^{s-1} dX dp 
\]
with $v > 0$ a volume term. \cite[\S 7.1 Proposition 1]{JaSh88} assures that the above integral is holomorphic at $s=1$.
\end{proof}

Let us now pay our attention to the non-vanishing of Jacquet-Shalika functional $\Lambda^{JS}(W_{\pi})$.

\begin{theorem}
\label{Shalika-unitary}
Let $\pi \in \mathcal{A}_F(m)$ be a unitary generic representation.
\begin{enumerate}[label=$(\mathrm{\alph*})$]
\item\label{JS-1} Suppose that $m=2n$ is even. Then we have
\[
\int_{N_n \backslash P_n} \int_{\mathcal{N}_n \backslash \mathcal{M}_n} W^{\circ}_{\pi} \left( \sigma_{2n} \begin{pmatrix} I_n & X \\ & I_n  \end{pmatrix} \begin{pmatrix} p & \\ & p \end{pmatrix} \right) \psi^{-1}({\rm Tr}(X)) dX dp
=
 \begin{cases}
  L(1,\pi_{ur},\wedge^2)& \text{if $\pi$ is ramified,} \\
  \dfrac{L(1,\pi,\wedge^2)}{L(n,\omega_{\pi})}& \text{otherwise.} 
  \end{cases}
\]
\item\label{JS-2} Suppose that $m=2n+1$ is odd. Then we obtain
\[
\int_{N_n \backslash G_n}  \int_{\mathcal{N}_n \backslash \mathcal{M}_n} W^{\circ}_{\pi} \left( \sigma_{2n+1} \begin{pmatrix} I_n &X& \\ &I_n& \\ &&1 \end{pmatrix} \begin{pmatrix} g && \\ &g& \\ &&1 \end{pmatrix} \right)\psi^{-1}({\rm Tr}(X)) dX dg=L(1,\pi_{ur},\wedge^2).
\]
\end{enumerate}
\end{theorem}

\begin{proof}
Evaluating the equality in Theorem \ref{JSEQ}-\ref{JSEQ-2} at $s=1$ demonstrates the second statement. The first assertion is a consequence of the proof of \cite[Lemma 3.4]{Jo20-2}, addressed in \cite[\S 3.2]{Grobner}. For the sake of completeness, we provide an alternative straightfoward approach. The unramified case is almost identical to the proof of \cite[Theorem 6.2]{AM} hence we omit the complete details. Mimicking the essential point made in the proof of Proposition \ref{Shalika-convergence}, the integral $J(s,W^{\circ}_{\pi},\Phi_c)$ turns out to be
\[
 J(s,W^{\circ}_{\pi},\Phi_c)=\int_{N_n \backslash P_n}  \int_{\mathcal{N}_n \backslash \mathcal{M}_n} W^{\circ}_{\pi} \left( \sigma_{2n} \begin{pmatrix} I_n & X \\ & I_n  \end{pmatrix} \begin{pmatrix} p & \\ & p \end{pmatrix} \right) \psi^{-1}({\rm Tr}(X))  |\mathrm{det}(p)|^{s-1} dX dp
\]
which is equivalent to $L(s,\pi_{ur},\wedge^2)$ as outlined in \ref{JSEQ}-\ref{JSEQ-1}. The result that we search for then follows from Proposition \ref{Shalika-convergence}, plugging in $1$ for $s$.
\end{proof}

As opposed to Theorems \ref{RS-priods} and \ref{Asai-FJ}, the right hand side of \ref{JS-1} and \ref{JS-2} in Theorem \ref{Shalika-unitary} takes different shapes,
depending on the presence of Schwartz-Bruhat functions in Jacquet-Shalika integrals \eqref{JS-odd} and \eqref{JS-even}.

\subsection{The $S_{2n}$-distinguished representation}
\label{JSS2n}

For this section, we restrict ourselves to the case for $m=2n$ even. We say that a representation $\pi$ of $G_{2n}$ is $(S_{2n},\Theta)$-{\it distinguished}, if $\mathrm{Hom}_{S_{2n}}(\pi,\Theta) \neq 0$. The central character $\omega_{\pi}$ of the $(S_{2n},\Theta)$-distinguished representation $\pi$ is always trivial. 

\par
For our convenience, we introduce an auxiliary symmetric square $L$-factors $\mathcal{L}(s,\pi,\mathrm{Sym}^2)$. It is proven in \cite[\S 5.2]{Jo20} that
$L(s,\pi,\wedge^2)^{-1}$ divides $L(s,\pi \times \pi)^{-1}$ in $\mathbb{C}[q^{\pm s}]$. Hence we can find a polynomial $Q(X) \in \mathbb{C}[X]$ satisfying $Q(0)=1$
and $L(s,\pi \times \pi)^{-1}=Q(q^{-s})L(s,\pi,\wedge^2)^{-1}$.
We define $\mathcal{L}(s,\pi,\mathrm{Sym}^2)$ by
\begin{equation}
\label{Ext-Sym-identity}
 \mathcal{L}(s,\pi,\mathrm{Sym}^2):=\frac{1}{Q(q^{-s})}.
\end{equation}

\begin{proposition}
\label{Sdistinction}
Let $\pi \in \mathcal{A}_F(2n)$ be a generic representation which is $(S_{2n},\Theta)$-distinguished. Then the $(P_{2n} \cap S_{2n},\Theta)$-invariant linear functional $\Lambda^{JS}(W_{\pi})$ is well-defined in that $L(s,\pi, \wedge^2)$ is holomorphic at $s=1$.
\end{proposition}

\begin{proof}
It is shown in \cite[Proposition 6.1]{JaRa} that if an irreducible representation $\pi$ of $G_{2n}$ is $S_{2n}$-distinguished, $\pi$ is self-dual. In the aspect of \eqref{Ext-Sym-identity}, the local Rankin-Selberg $L$-function enjoys a 
factorization 
$L(s,\pi \times \pi)=\mathcal{L}(s,\pi,\mathrm{Sym}^2)L(s,\pi,\wedge^2)$. 
As a result, $L(s,\pi, \wedge^2)$ is holomorphic at $s=1$ as soon as $L(s,\pi \times \pi)$ is so. However this will be the case, according to Proposition \ref{Bernstein-s=1}.
\end{proof}

As mentioned earlier, we can ease the unitarity restriction in Theorem \ref{Shalika-unitary} if we assume that $\pi$ is $(S_{2n},\Theta)$-distinguished.

\begin{theorem}
\label{MirabolicShalika}
Let $\pi \in \mathcal{A}_F(2n)$ be a generic representation which is distinguished with respect to $(S_{2n},\Theta)$. Then we have
\[
\int_{N_n \backslash P_n} \int_{\mathcal{N}_n \backslash \mathcal{M}_n} W^{\circ}_{\pi} \left( \sigma_{2n} \begin{pmatrix} I_n & X \\ & I_n  \end{pmatrix} \begin{pmatrix} p & \\ & p \end{pmatrix} \right) \psi^{-1}({\rm Tr}(X)) dX dp
=
 \begin{cases}
  L(1,\pi_{ur},\wedge^2)& \text{if $\pi$ is ramified,} \\
  \dfrac{L(1,\pi,\wedge^2)}{L(n,{\bf 1}_{F^{\times}} )}& \text{otherwise.}  
  \end{cases}
\]
\end{theorem}

We record the straightforward consequence from the proof of \cite[Proposition 3.4]{Jo20} that is built on \cite{MA15}.

\begin{proposition}
\label{Shalika-one-dimension}
Let $\pi \in \mathcal{A}_F(2n)$ be a unitary representation. Then
\[
  {\rm dim}_{\mathbb{C}} {\rm Hom}_{P_{2n} \cap S_{2n}}(\pi,\Theta) \leq 1.
\]
The equality holds when $\pi$ is generic. In particular, if $\pi$ is $(S_{2n},\Theta)$-distinguished, then $W_{\pi} \mapsto \Lambda^{JS}(W_{\pi} )$ gives a unique non-trivial $S_{2n}$-quasi-invariant linear functional belonging to the space ${\rm Hom}_{P_{2n} \cap S_{2n}}(\pi,\Theta)$.
\end{proposition}

\begin{proof}
Thanks to \cite[Proposition 4.3]{Ma14JNT}, the space ${\rm Hom}_{P_{2n} \cap S_{2n}}(\pi,\Theta)$ embeds as the subspace of ${\rm Hom}_{P_{2n} \cap M_{2n}}(\pi,{\bf 1})$. Conjugating by $w_{2n}$ provides the isomorphism  
\[
{\rm Hom}_{P_{2n} \cap M_{2n}}(\pi,{\bf 1}_{P_{2n} \cap M_{2n}}) \simeq {\rm Hom}_{P_{2n} \cap H_{2n}}(\pi,{\bf 1}_{P_{2n}\cap H_{2n}}).
\]
We kindly refer the reader to \S \ref{FJperiod} for the exact definitions of $w_{2n}$ and $H_{2n}$. It is evident from the proof of \cite[Corollary 4.2]{MA15} that the latter space ${\rm Hom}_{P_{2n} \cap H_{2n}}(\pi,{\bf 1}_{P_{2n}\cap H_{2n}})$ has dimension at most one (cf. Proposition \ref{linear-dimone}). Keeping in mind the assumption ${\rm Hom}_{S_{2n}}(\pi,\Theta) \neq 0$, the inclusion ${\rm Hom}_{S_{2n}}(\pi,\Theta) \subseteq {\rm Hom}_{P_{2n} \cap S_{2n}}(\pi,\Theta) $ induces the isomorphism 
\[
{\rm Hom}_{P_{2n} \cap S_{2n}}(\pi,\Theta) \simeq {\rm Hom}_{S_{2n}}(\pi,\Theta)
\]
which we utilize to construct a non-trivial $S_{2n}$-quasi-invariant linear functional, from Theorem \ref{MirabolicShalika}.
\end{proof}


\subsection{The Friedberg-Jacquet (linear) period}
\label{FJperiod}

We briefly remind the reader of the integral representation introduced by Bump and Friedberg \cite{BF,MY}.
We define the embedding $J: G_n \times G_{n} \rightarrow G_m$ by
\[
 J(g,g^{\prime})_{k,l}=
  \begin{cases}
  g_{i,j} &  \text{if $k=2i-1$, $l=2j-1$, } \\
  g^{\prime}_{i,j} &  \text{if $k=2i$, $l=2j$,} \\
 0 &  \text{otherwise,} \\
  \end{cases}
\]
for $m=2n$ even
and  $J: G_{n+1} \times G_n \rightarrow G_m$ by
\[
 J(g,g^{\prime})_{k,l}=
  \begin{cases}
  g_{i,j} &  \text{if $k=2i-1$, $l=2j-1$, } \\
  g^{\prime}_{i,j} &  \text{if $k=2i$, $l=2j$,} \\
 0 &  \text{otherwise,} \\
  \end{cases}
\]
for $m=2n+1$ odd. As for the purpose of holding onto coherent terminology with Matringe \cite{MA15,MA17}, the reader should perceive that the role of $g$ and $g'$ in \cite{BF,MY} is swapped for even cases. The test vector problem for Bump-Friedberg exterior square $L$-factor has been settled in \cite[Theorem 5.1]{MY}.

\begin{theorem}[Miyauchi and Yamauchi] Let $\pi \in \mathcal{A}_F(m)$ be a generic representation. We set $c=c(\pi)$.
\label{BFEQ}
\begin{enumerate}[label=$(\mathrm{\alph*})$]
\item\label{BFL-1} Suppose that $m=2n$ is even. Then we have
\begin{equation}
\label{BF-original-even}
\begin{split}
&L(s_1,\pi_{ur})L(s_2,\pi_{ur},\wedge^2)\\
&=\int_{N_{n} \backslash G_{n}} \int_{N_n \backslash G_n} W^{\circ}_{\pi}(J(g,g^{\prime})) \Phi_c(e_ng') |\mathrm{det}(g)|^{s_1-1/2} |\mathrm{det}(g^{\prime})|^{1/2+s_2-s_1} dg dg^{\prime}.
\end{split}
\end{equation}
\item\label{BFL-2} Suppose that $m=2n+1$ is odd. Then we obtain
\begin{equation}
\label{BF-original-odd}
\begin{split}
&L(s_1,\pi_{ur})L(s_2,\pi_{ur},\wedge^2)\\
&=\int_{N_{n} \backslash G_{n}} \int_{N_{n+1} \backslash G_{n+1}} W^{\circ}_{\pi}(J(g,g^{\prime})) \Phi_c(e_{n+1}g) |\mathrm{det}(g)|^{s_1} |\mathrm{det}(g^{\prime})|^{s_2-s_1} dg dg^{\prime}.
\end{split}
\end{equation}
\end{enumerate}
\end{theorem}

To be more compatible with the standard language, we retain notations from Matringe \cite{MA15,MA17}. For $m=2n$ even, we denote by $M_{2n}$ the standard Levi of $G_{2n}$ associated with the partition $(n,n)$ of $2n$. Let $w_{2n}=\sigma_{2n}$ and then we set $H_{2n}=w_{2n}M_{2n}w^{-1}_{2n}$. Let $w_{2n+1}=w_{2n+2}|_{G_{2m+1}}$ so that
\[
 w_{2n+1}=\begin{pmatrix} 1 & 2 & \dotsm & n+1 & | & n+2 & n+3 & \dotsm &2n & 2n+1 \\ 
                                                1 & 3 & \dotsm & 2n+1 & | &  2 & 4 & \dotsm & 2n-2 &2n \\ 
                                                \end{pmatrix}.
\]
In the odd case, $w_{2n+1} \neq \sigma_{2n+1}$ and we let $M_{2n+1}$ denote the standard Levi associated to the partition $(n+1,n)$ of $2n+1$. We set $H_{2n+1}=w_{2n+1}M_{2n+1}w^{-1}_{2n+1}$. We note that $H_m$ is compatible in the sense that $H_m \cap G_{m-1}=H_{m-1}$ and we can easily see that $J(g,g')=w_m \mathrm{diag}(g,g') w^{-1}_m$ for $\mathrm{diag}(g,g') \in M_m$. If $\alpha$ is a real number, we denote by $\chi_{\alpha}$ the character 
\[
  \chi_{\alpha} : J(g,g^{\prime}) \mapsto \left| \frac{\mathrm{det}(g)}{\mathrm{det}(g^{\prime})} \right|^{\alpha}
\]
of $H_m$. We denote by $\chi_m$ and $\mu_m$ characters of $H_m$;
\[
  \chi_m(J(g,g'))=
  \begin{cases}
{\bf 1}_{H_m} & \text{for $m=2n$,} \\
\displaystyle  \left| \frac{\mathrm{det}(g)}{\mathrm{det}(g^{\prime})} \right| & \text{for $m=2n+1$,} \\
\end{cases}
\quad
 \mu_m(J(g,g'))=
  \begin{cases}
\displaystyle  \left| \frac{\mathrm{det}(g)}{\mathrm{det}(g^{\prime})} \right| & \text{for $m=2n$.} \\
{\bf 1}_{H_m} & \text{for $m=2n+1$,} \\
\end{cases}
\]
We now turn toward the case for $s_1=s$ and $s_2=2s$. By taking all these accounts, primary Bump-Friedberg integrals \eqref{BF-original-even} and \eqref{BF-original-odd} can be unified as one single integral $\mathcal{Z}(s,\chi_{-1/2}, W_{\pi},\Phi)$ of the form;
\[
\begin{split}
  &\mathcal{Z}(s,\chi_{-1/2}, W_{\pi},\Phi)
  =\int_{(N_m \cap H_m) \backslash H_m} W_{\pi}(h)\chi_{-1/2}(h)\chi_m^{1/2}(h) \Phi(l_m(h)) |\mathrm{det}(h)|^s dh\\
  &=\begin{cases}
  \displaystyle
  \int_{N_{n} \backslash G_{n}} \int_{N_n \backslash G_n} W_{\pi}(J(g,g^{\prime})) \Phi(e_ng')  \left|\frac{\mathrm{det}(g)}{\mathrm{det}(g^{\prime})} \right|^{-\frac{1}{2}}\ |\mathrm{det}(gg')|^s dg dg^{\prime} & \text{for $m=2n$,} \\
   \displaystyle
  \int_{N_{n} \backslash G_{n}} \int_{N_{n+1} \backslash G_{n+1}} W_{\pi}(J(g,g^{\prime})) \Phi(e_{n+1}g) |\mathrm{det}(gg')|^s dg dg^{\prime} & \text{for $m=2n+1$,}
\end{cases}
\end{split}
\]
where $W_{\pi} \in \mathcal{W}(\pi,\psi)$, $\Phi \in \mathcal{S}(F^{[(m+1)/2]})$, and $l_m(J(g,g'))$ is $e_ng'$ for $m=2n$ even and $e_{n+1}g$ for $m=2n+1$ odd. The integral $\mathcal{Z}(s,\chi_{-1/2}, W_{\pi},\Phi)$ converges absolutely for $s$ of real part large enough. The $\mathbb{C}$-vector space generated by the local Bump-Friedberg integrals 
\[
 \langle \mathcal{Z}(s,\chi_{-1/2}, W_{\pi},\Phi)\;|\; W_{\pi} \in \mathcal{W}(\pi,\psi), \Phi \in \mathcal{S}(F^{[(m+1)/2]}) \rangle
\]
is in fact a $\mathbb{C}[q^{\pm s}]$-fractional ideal $\mathcal{I}(\pi,\chi_{-1/2},\wedge^2)$ of $\mathbb{C}(q^{-s})$. The ideal $\mathcal{I}(\pi,\chi_{-1/2},\wedge^2)$ is principal, and has a unique generator of the form $P(q^{-s})^{-1}$, where $P(X)$ is a polynomial in $\mathbb{C}[X]$ with $P(0)=1$. The local Bump-Friedberg $L$-function associated to $\pi$ is defined by the unique normalized generator;
  \[
   L^{BF}(s,\pi,\chi_{-1/2})=\frac{1}{P(q^{-s})}.
  \]
The formal local $L$-function $L(s,\pi_{ur})(2s,\pi_{ur},\wedge^2)$ does not always agree with $L^{BF}(s,\pi,\chi_{-1/2})$.

\begin{corollary} Let $\pi \in \mathcal{A}_F(m)$ be a generic representation. Then we have
\[
 L(s,\pi_{ur})(2s,\pi_{ur},\wedge^2)=P(q^{-s})L^{BF}(s,\pi,\chi_{-1/2})
\]
for a polynomial $P(X) \in \mathbb{C}[X]$ satisfying $P(0)=1$.
\end{corollary}

\begin{proof}
We only need to check from Theorem \ref{BFEQ} that $ L(s,\pi_{ur})(2s,\pi_{ur},\wedge^2)$ which is the sam as $\mathcal{Z}(s,\chi_{-1/2}, W^{\circ}_{\pi},\Phi_c)$ in turn belongs to the $\mathbb{C}[q^{\pm s}]$-fractional ideal $\mathcal{I}(\pi,\chi_{-1/2},\wedge^2)$ of $\mathbb{C}(q^{-s})$.
\end{proof}

The following integral
\[
  \Lambda^{FJ}(W_{\pi}):=\int_{(N_m \cap H_m) \backslash (P_m \cap H_m) }W_{\pi}(p)\chi_{-1/2}(p)\mu_m^{1/2}(p) dp.
\]
makes sense at least formally. The following proposition gives a meaning to this integral.

\begin{proposition}
Let $\pi \in \mathcal{A}_F(m)$ be a unitary generic representation.  For any $W_{\pi} \in \mathcal{W}(\pi,\psi)$, the integral 
\[
\int_{(N_m \cap H_m) \backslash (P_m \cap H_m) }W_{\pi}(p)\chi_{-1/2}(p)\mu_m^{1/2}(p) |\mathrm{det}(p)|^{s-1/2} dp
  \]
  converges absolutely for the closed right half-plane $\mathrm{Re}(s) \geq 1/2$. In particular, for $m=2n$ even, $W_{\pi} \mapsto \Lambda^{FJ}(W_{\pi})$ defines a $P_{2n} \cap H_{2n}$-invariant linear functional on $\mathcal{W}(\pi,\psi)$.

\end{proposition}

\begin{proof}
It is a consequence of \cite[Proposition 4.7]{MA15} and Bernstein's criterion \cite[\S 7.3]{Bern} for the exponent of central characters that can be taken verbatim from the proof of Proposition \ref{RS-multiplicityone}.
\end{proof}

We now take up the issue of the non-vanishing of the Friedberg-Jacquet linear functional on the Whittaker model $\mathcal{W}(\pi,\psi)$.

\begin{theorem}
\label{Hm-unitary}
Let $\pi \in \mathcal{A}_F(m)$ be a unitary generic representation. Then we have
\[
  \int_{(N_m \cap H_m) \backslash (P_m \cap H_m) }W^{\circ}_{\pi}(p)\chi_{-1/2}(p)\mu_m^{1/2}(p)  dp=
   \begin{cases}
    L(1/2,\pi_{ur})L(1,\pi_{ur},\wedge^2)
   & \text{if $\pi$ is ramified,} \\
 \dfrac{L(1/2,\pi)L(1,\pi,\wedge^2)}{L(m/2,\omega_{\pi})}  & \text{otherwise.} 
  \end{cases}
\]
\end{theorem}

\begin{proof} 
We deal with the case $m=2n$ even and $\pi$ ramified. Exploiting the partial Iwasawa decomposition $G_{n-1}=N_{n-1}A_{n-1}K_{n-1}$, we obtain
\[
\Lambda^{FJ}(W^{\circ}_{\pi} )=\int_{A_{n-1}} \int_{A_n} W^{\circ}_{\pi} \left( J\left( a, \begin{pmatrix} a' & \\ & 1 \end{pmatrix} \right) \right) \delta_{B_n}^{-1}(a) \delta_{B_{n-1}}^{-1}(a')  da da'. 
\]
We first deal with the case $r$ even. In virtue of Theorem \ref{NewV}, we insert the expression of $W^{\circ}_{\pi}$ into the above integral
\[
\begin{aligned}
\Lambda^{FJ}(W^{\circ}_{\pi} )
=&\int_{A_{r\slash 2}} \int_{A_{r\slash 2}} W^{\circ}_{\pi_{ur}} (J(b,b'))\delta_{B_{n}}^{-1}\begin{pmatrix} b & \\ & I_{n-r/2}\end{pmatrix} \delta_{B_{n-1}}^{-1}\begin{pmatrix} b^{\prime } & \\ & I_{n-r/2-1}\end{pmatrix}\nu^{\frac{m-r}{2}}(bb^{\prime})  \\
&\times {\bf 1}_{\mathcal{O}}(b^{\prime}_{r/2}) db db' \\
=&\int_{A_{r\slash 2}} \int_{A_{r\slash 2}} W^{\circ}_{\pi_{ur}} (J(b,b'))\delta_{B_{r/2}}^{-1}(b) \delta_{B_{r/2}}^{-1}(b')\nu^{\frac{m-r}{2}}(bb^{\prime})
\nu^{r/2-n}(b)\nu^{r/2-(n-1)}(b')\\
&\times {\bf 1}_{\mathcal{O}}(b^{\prime}_{r/2}) db db'. \\
  \end{aligned}
  \]
  But we have the identity $\delta_{B_{r/2}}^{-1}(b) \delta_{B_{r/2}}^{-1}(b') \nu^{-\frac{1}{2}}(b)\nu^{\frac{1}{2}}(b')=\delta_{B_{r}}^{-\frac{1}{2}}(J(b,b'))$ which gives rise to
  \begin{equation}
  \label{BF-even}
\Lambda^{FJ}(W^{\circ}_{\pi}) = \int_{A_{r\slash 2}} \int_{A_{r\slash 2}} W^{\circ}_{\pi_{ur}} (J(b,b'))\delta_{B_{r}}^{-\frac{1}{2}}(J(b,b')) {\bf 1}_{\mathcal{O}}(b^{\prime}_{r/2})  |\mathrm{det}(bb')|^{\frac{1}{2}} db db'.
  \end{equation}
 Provided that $r$ is odd, Theorem \ref{NewV} leads us to the integral of the form
\[
\begin{aligned}
\Lambda^{FJ}(W^{\circ}_{\pi})
 &=\int_{A_{(r+1)\slash 2}} \int_{A_{(r-1)\slash 2}} W^{\circ}_{\pi_{ur}} (J(b,b'))\delta_{B_{n}}^{-1}\begin{pmatrix} b & \\ & I_{n-(r+1)/2}\end{pmatrix} \delta_{B_{n-1}}^{-1}\begin{pmatrix} b^{\prime } & \\ & I_{n-(r-1)/2-1}\end{pmatrix} \\
 &\quad \times  \nu^{\frac{m-r}{2}}(bb^{\prime}) {\bf 1}_{\mathcal{O}}(b_{(r+1)/2}) db' db \\
 &=\int_{A_{(r+1)\slash 2}} \int_{A_{(r-1)\slash 2}} W^{\circ}_{\pi_{ur}} (J(b,b'))\delta_{B_{(r+1)/2}}^{-1}(b) \delta_{B_{(r-1)/2}}^{-1}(b')
\nu^{(r+1)/2-n}(b)\nu^{(r-1)/2-(n-1)}(b')\\
&\quad \times \nu^{\frac{m-r}{2}}(bb^{\prime}){\bf 1}_{\mathcal{O}}(b_{(r+1)/2}) db' db.  \\
 \end{aligned}
 \]
 Using the relation $\delta_{B_{(r+1)/2}}^{-1}(b) \delta_{B_{(r-1)/2}}^{-1}(b')=\delta_{B_{r}}^{-\frac{1}{2}}(J(b,b'))$, we see that
\begin{equation}
\label{BF-odd}
 \Lambda^{FJ}(W^{\circ}_{\pi})
 =\int_{A_{(r+1)\slash 2}} \int_{A_{(r-1)\slash 2}} W^{\circ}_{\pi_{ur}} (J(b,b'))\delta_{B_{r}}^{-\frac{1}{2}}(J(b,b')) {\bf 1}_{\mathcal{O}}(b_{(r+1)/2})
|\mathrm{det}(bb')|^{\frac{1}{2}} db' db.  
\end{equation}
Further, \eqref{BF-even} combined with \eqref{BF-odd} yields that the integral $\Lambda^{FJ}(W^{\circ}_{\pi})$ is in accord with the expression for $\mathcal{Z}(1/2, \chi_{-1/2}, W^{\circ}_{\pi_{ur}},{\bf 1}_{\mathcal{O}^r} )$. It is easy from \citelist{\cite{BF}*{\S 3}\cite{MY}*{Theorem 5.1}} to verify that the integral $\mathcal{Z}(1/2, \chi_{-1/2}, W^{\circ}_{\pi_{ur}},{\bf 1}_{\mathcal{O}^r} )$ is nothing else but $L(1/2,\pi_{ur})L(1,\pi_{ur},\wedge^2)$ that we search for. A similar process applies to the $m=2n+1$ odd case. The unramified case proceeds as in the proof of \cite[Theorem 6.2]{AM} hence we omit the complete details.
\end{proof}

\subsection{The $H_{2n}$-distinguished representation}
\label{BJH2n}

For this section, we only concentrate on the case for $m=2n$ even. We say that $\Delta=[\rho\nu^{1-\ell},\rho\nu^{2-\ell}, \dotsm,\rho]$ {\it precedes} $\Delta'=[\rho'\nu^{1-\ell'},\rho'\nu^{2-\ell'}, \dotsm,\rho']$ if $\rho'=\rho\nu^i$ for 
some $\max\{1, \ell'-\ell+1 \} \leq i \leq \ell'$, which we denote by $\Delta \prec {\Delta'}$. We say that $\Delta$ and $\Delta'$ are {\it linked} if either $\Delta \prec {\Delta'}$ or $\Delta' \prec \Delta$, and {\it non-linked} otherwise. According to \cite[Theorem 9.7]{Zel80}, $\pi  \in \mathcal{A}_F(n)$ is a generic representation if and only if there exist unlinked segments $\Delta_1, \Delta_2, \dotsm, \Delta_t$ such that $\pi  \simeq {\rm Ind}^{G_n}_{\rm P}(\Delta_1 \otimes \Delta_2  \otimes \dotsm \otimes \Delta_t )$ with the induction being normalized parabolic induction.
We embark on the following characterization of poles \cite[Lemma 2.3]{MO} which is a consequence of \cite[Proposition 8.2]{JPSS3}.

\begin{lemma}
\label{M-O}
Let $\Delta, \Delta' \in \mathcal{A}_F$ be quasi-square integrable representations. Then 
$L(s,\Delta \times \Delta')$ has a pole at $s=1$ if and only if $\Delta \prec {\Delta'}^{\vee}$. 
\end{lemma}

We recall from \S \ref{FJperiod} that $H_{2n}=w_{2n}M_{2n}w^{-1}_{2n}$, where $M_{2n}$ is the standard Levi subgroup associated with the partition $(n,n)$
and $w_{2n}=\sigma_{2n}$. We say that a representation $\pi$ of $G_{2n}$ is $H_{2n}$-{\it distinguished}, if  $\mathrm{Hom}_{H_{2n}}(\pi,{\bf 1}) \neq 0$.
As commented earlier, the unitarity hypothesis in Theorem \ref{Hm-unitary} can be removed if we assume that $\pi$ is $H_{2n}$-distinguished. The central character $\omega_{\pi}$ of the $H_{2n}$-distinguished representation $\pi$ is always trivial.

\begin{proposition}
Let $\pi \in \mathcal{A}_F(2n)$ be a generic representation which is $H_{2n}$-distinguished. Then the $P_{2n} \cap H_{2n}$-invariant linear functional $W_{\pi} \mapsto \Lambda^{FJ}(W_{\pi})$ is well-defined in that  $L^{BF}(s,\pi,\chi_{-1/2})$ is holomorphic at $s=1/2$.
\end{proposition}

\begin{proof}
The main result of \cite{Jo20} together with \cite{MA15,MA17} tells us that
 \[
 L^{BF}(s,\pi,\chi_{-1/2})=L(s,\pi)L(2s,\pi,\wedge^2).
 \]
The multiplicative relation of the Bump-Friedberg $L$-factors has been achieved by Matringe \cite[Theorem 5.1]{MA15};
\[
 L^{BF}(s,\pi,\chi_{-1/2})=\prod_{k=1}^t L(s,\Delta_k)\prod_{k=1}^t L(2s,\Delta_k,\wedge^2) \prod_{1 \leq i < j \leq t} L(2s, \Delta_i \times \Delta_j).
\]
Indeed, if $L^{BF}(s,\pi,\chi_{-1/2})$ admits a pole at $s=1/2$, then either $L(s,\Delta_k)$ or $L(2s,\Delta_k,\wedge^2)$ has a pole at $s=1/2$ for some $k$, or $L(s, \Delta_i \times \Delta_j)$ has a pole at $s=1$ for some $(i,j)$. 

\par
We first suppose that $L(s, \Delta_i \times \Delta_j)$ has a pole at $s=1$. Then by Lemma \ref{M-O} we have $\Delta_i \prec \Delta^{\vee}_j$, but by our assumption that $\pi$ is $H_{2n}$-distinguished, we know from \cite[Theorem 1.1]{JaRa} (cf. see the paragraph after \cite[Definition 2.1]{MA15}) that $\pi$ is self-dual. As a result, $\Delta^{\vee}_j$ is $\Delta_l$ for some $l \neq i$ which contradicts the fact that $\Delta_i$'s are unlinked. In the spirit of Proposition \ref{Sdistinction}, we can exclude the case when $L(2s,\Delta_k,\wedge^2)$ has a pole at $s=1/2$. We observe from Theorem \ref{G-J} that $L(s,\Delta_k) \equiv 1$ unless $\Delta_k$ is an unramified character of $F^{\times}$. This would in turn imply that $L(s,\Delta_k)$ is holomorphic at $s=1/2$ except $\Delta_k=[\nu^{1/2-\ell},\nu^{3/2-\ell}, \dotsm,\nu^{-1/2}]$, $\ell \geq 1$. Although this might be the case, we still reach the exactly same contradiction, because $\Delta_{k}^{\vee}=[\nu^{1/2},\dotsm,\nu^{\ell-3/2},\nu^{\ell-1/2}]$ is $\Delta_{l'}$ for some $l' \neq k$ which makes $\Delta_i$'s linked.
\end{proof}

It is worthwhile noting that for $s=1/2$ and $m=2n$ even, the integral
\[
 \int_{(N_m \cap H_m) \backslash (P_m \cap H_m) }W_{\pi}(p)\chi_{-1/2}(p)\mu_m^{1/2}(p) |\mathrm{det}(p)|^{s-1/2} dp
 =\int_{(N_{2n} \cap H_{2n}) \backslash (P_{2n} \cap H_{2n}) } W_{\pi}(p)  dp
\]
corresponds to what is introduced as the non-trivial $H_{2n}$-linear functional (up to conjugation) in \cite[Proposition 3.1]{LapidMao}.

\begin{theorem}
\label{H2n-generic-distinguished}
Let $\pi \in \mathcal{A}_F(2n)$ be a generic representation which is distinguished with respect to $H_{2n}$. Then we have
\[
 \int_{(N_{2n} \cap H_{2n}) \backslash (P_{2n} \cap H_{2n}) }W^{\circ}_{\pi}(p) dp=
  \begin{cases}
    L(1/2,\pi_{ur})L(1,\pi_{ur},\wedge^2)
   & \text{if $\pi$ is ramified,} \\
 \dfrac{L(1/2,\pi)L(1,\pi,\wedge^2)}{L(n,{\bf 1}_{F^{\times}} )}  & \text{otherwise.} 
  \end{cases}
\]
\end{theorem}

The proof of \cite[Corollary 4.2]{MA15} implicitly contains the following assertion.

\begin{proposition} 
\label{linear-dimone}
Let $\pi \in \mathcal{A}_F(2n)$ be a unitary representation. Then
\[
  {\rm dim}_{\mathbb{C}} {\rm Hom}_{P_{2n} \cap H_{2n}}(\pi,{\bf 1}_{P_{2n} \cap H_{2n}}) \leq 1.
\]
The equality holds when $\pi$ is generic. In particular, if $\pi$ is $H_{2n}$-distinguished, then $W_{\pi} \mapsto \Lambda^{BF}(W_{\pi} )$ gives a unique non-trivial $H_{2n}$-invariant linear functional in the space  ${\rm Hom}_{H_{2n}}(\pi,{\bf 1}_{H_{2n}})$.
\end{proposition}

\begin{proof}
As alluded in the course of the proof of Proposition \ref{Shalika-one-dimension}, on account of the assumption ${\rm Hom}_{P_{2n} \cap H_{2n}}(\pi,{\bf 1}) \neq 0$, the embedding  
${\rm Hom}_{H_{2n}}(\pi,{\bf 1}_{H_{2n}}) \hookrightarrow {\rm Hom}_{P_{2n} \cap H_{2n}}(\pi,{\bf 1}_{P_{2n}\cap H_{2n}})$ induces an isomorphism 
\[
{\rm Hom}_{P_{2n} \cap H_{2n}}(\pi,{\bf 1}_{P_{2n} \cap H_{2n}}) \simeq  {\rm Hom}_{H_{2n}}(\pi,{\bf 1}_{H_{2n}})
\]
from which a non-trivial $P_{2n} \cap H_{2n}$-invariant linear functional in Theorem \ref{H2n-generic-distinguished} can be deemed the $H_{2n}$-invariant form.
\end{proof}


\section{The Bump-Ginzburg Period}
\label{BG}

From this section onward we declare that the characteristic of residue field of $F$ is odd. We do not strive for maximal generality, so sometime the hypothesis might not be necessary but which hold in all our applications. For the reader, who seeks for an expository description of symmetric square $L$-factors, we reiterate the crucial points from \cites{Takeda14,Yamana}. 
Banks, Levy, and Sepanski \cite[Section 3]{BLS} gave an explicit illustration of a 2-cocycle,
\[
  \sigma_m : G_m \times G_m \rightarrow \{ \pm 1 \}
\]
and the 2-cocycles $\{ \sigma_m \}_{m=1}^{\infty}$ satisfy a block compatibility formula on all standard Levi subgroups; if $m=m_1+m_2+\dotsm+m_t$ and $g_i, g'_i \in G_{m_i}$ for all $i=1,2,\dotsm,t$, then
\[
 \sigma_m \left( \begin{pmatrix} g_1 \vspace{-1ex} && \\ & \ddots& \vspace{-1ex} \\ &&g_t \end{pmatrix},\begin{pmatrix} g'_1 \vspace{-1ex} && \\ & \ddots& \vspace{-1ex} \\ &&g'_t \end{pmatrix}  \right)
 =\prod_{k=1}^t \sigma_{m_k}(g_k,g'_k) \prod_{1 \leq i < j \leq t} (\mathrm{det}(g_i), \mathrm{det}(g'_j)),
\]
where $(-,-)$ stands for the Hibert symbol for $F$ (We apologize the double use of $\sigma_m$ but we hope that the reader can separate them by context). The 2-cocycle $\sigma_1$ is trivial on $G_1$ and the 2-cocycle $\sigma_2$ is well-known as the Kubota 2-cocycle on $G_2$. We denote the central extension of $G_m$ associated to $\sigma_m$ by $^{\sigma_m}\widetilde{G}_m$. As a set, the two-fold cover $^{\sigma_m}\widetilde{G}_m$ is incarnated as $G_m \times \{ \pm 1 \}$ and the group law is defined by 
\[
  (g,\xi)\cdot(g',\xi')=(gg',\sigma_m(g,g')\xi\xi')  \;\;  \text{for} \;\; \xi,\xi' \in \{\pm 1 \} \;\; \text{and} \;\; g,g' \in G_m.
  \]
It is known that the maximal compact open subgroup $K_m$ splits in $^{\sigma_m}\widetilde{G}_m$ and hence we can find a continuous map $s_m : G_m \rightarrow \{ \pm 1 \}$ such that
\[
  \sigma_m(k,k')=s_m(k)s_m(k')s_m(kk')  \;\;  \text{for} \;\; k,k' \in K_m.
\]
For the global application, we need to use a different 2-cocycle $\tau_m$ which satisfies $\tau_m(k,k')=1$ for $k,k' \in K_m$. Then we define the 2-cocycle $\tau_m$ by
\[
 \tau_m(g,g')=\sigma_m(g,g')s_m(g)s_m(g')s_m(gg')
\]
for $g,g' \in G_m$. The choice of $s_m$ and $\tau_m$ is not unique. We shall fix $s_m$ consistent with ones in \cite{Takeda14,Yamana}, which all stem from what is called the canonical lift of Kazhdan and Patterson. We define the metaplectic cover of $G_m$ to be $\widetilde{G}_m=G_m \rtimes \{ \pm 1\}$ which sits in the central extension of $G_m$ associated to $\tau_m$
\[
  1  \longrightarrow \{ \pm 1 \} \longrightarrow \widetilde{G}_m \longrightarrow G_m \longrightarrow 1
\]
and the group law is given by
\[
  (g,\xi)\cdot(g',\xi')=(gg',\tau_m(g,g')\xi\xi')   \;\;  \text{for} \;\; \xi,\xi' \in \{\pm 1 \} \;\; \text{and} \;\; g,g' \in G_m.
  \]
We define the canonical projection $p_m : \widetilde{G}_m \rightarrow G_m$ by $p_m((g,\xi))=g$ for $g \in G_m$ and $\xi \in \{ \pm 1 \}$.
For any subgroup $H$ of $G_m$, we write $\widetilde{H}$ for its preimage $p_m^{-1}(H)$. We define a set theoretic section ${\bf s} : G_m \rightarrow \widetilde{G}_m$ by ${\bf s}(g)=(g,s_m(g))$ for $g \in G_m$. Any element can be uniquely written in the form $\xi {\bf s}(g)$ for $g \in G_m$ and $\xi \in \{ \pm 1 \}$ and it is worthwhile noting that
\[
  {\bf s}(g){\bf s}(g')={\bf s}(gg') \sigma_m(g,g')
\]
for $g, g' \in G_m$. We define another set theoretic section $\kappa : G_m \rightarrow \widetilde{G}_m$ by $\kappa(g)=(g,1)$ for $g \in G_m$. 
For a subgroup $H \leq  G_m$, whenever the cocycle $\sigma_m$ is trivial on $H \times H$, the section $\bf s$ splits $H$, and we denote by $H^{\ast}$ the image ${\bf s}(H)$.

\par
Let $\pi$ be an admissible representation of a subgroup $\widetilde{H} \leq \widetilde{G}_m$. The representation $\pi$ is said to be {\it genuine} if $\pi(\xi {\bf s}(h))v=\xi \pi({\bf s}(h))v$
for all $h \in H$, $\xi \in \{ \pm 1 \}$, and $v \in V$, that is, each element in $\xi \in \widetilde{H}$ acts as a multiplication by $\xi$. On the other hand, any representation $\pi$ of $H$ can be pulled back to a non-genuine representation of $\widetilde{H}$ by composing it with the canonical projection $p_m : \widetilde{G}_m \rightarrow G_m$. In particular, for an upper parabolic subgroup $\rm P$, we view the modulus character $\delta_{\rm P}$ as an character on $\widetilde{\rm P}$ in this way.

\par
Let $A_m^{\rm e}$ be the subgroup of $A_m$ consisting of diagonal matrices $\mathrm{diag}(a_1,a_2,\dotsm,a_m)$ such that
\[
  \begin{cases}
   a_1a^{-1}_2, a_3a^{-1}_4, \dotsm, a_{m-1}a^{-1}_m \;\; \text{are squares,} & \text{if $m=2n$ is even,}\\
    a_2a^{-1}_3, a_4a^{-1}_5, \dotsm, a_{m-1}a^{-1}_m \;\; \text{are squares,} & \text{if $m=2n+1$ is odd.}\\
    \end{cases}
\]
We put $\mathscr{Z}_m=Z_m^{e(m)}$, where $e(m)=1$ or $2$ according as $m$ is odd or even. We remark that $\widetilde{\mathscr{Z}}_m$ is the center of $\widetilde{G}_m$. We set $\mu_{\psi}(a)=\gamma(\psi_a) / \gamma(\psi)$ for $a \in F^{\times}$, where $\psi_a(x)=\psi(ax)$ and $\gamma(\psi)$ is the Weil representation associated to $\psi$. We recall that
\[
 \mu_{\psi}(ab)=\mu_{\psi}(a)\mu_{\psi}(b)(a,b), \quad \mu_{\psi}(ab^2)=\mu_{\psi}(a)  \;\;  \text{for} \;\; a, b \in F^{\times}.
\]
These properties allow us to define a genuine character $\omega^{\psi}$ of $\widetilde{A}^{\rm e}_m$ by
\[
 \omega^{\psi}(\xi {\bf s}(a))=
 \begin{cases}
  \xi \mu^{-1}_{\psi}(a_m)\mu^{-1}_{\psi}(a_{m-2})\dotsm \mu^{-1}_{\psi}(a_2),    & \text{if $m=2n$ is even,}\\
   \xi \mu^{-1}_{\psi}(a_m)\mu^{-1}_{\psi}(a_{m-2})\dotsm \mu^{-1}_{\psi}(a_3),  & \text{if $m=2n+1$ is odd.}\\
    \end{cases}
\]
The {\it exceptional representation} $\theta^{\psi}_m$ of Kazhdan and Patterson is defined to be the unique irreducible quotient of the induced representation $\mathrm{Ind}^{\widetilde{G}_m}_{\widetilde{T}_m^{\rm e}N^{\ast}_m}(\omega^{\psi} \otimes \delta_{B_m}^{1/4})$ \cite[\S 2.1]{Takeda14}, where the induction is normalized in order that $\mathrm{Ind}^{\widetilde{G}_m}_{\widetilde{T}_m^{\rm e}N^{\ast}_m}(\omega^{\psi} \otimes \delta_{B_m}^{1/4})$ is unitarizable provided that $\omega^{\psi} \otimes \delta_{B_m}^{1/4}$. The representation $\theta^{\psi}_m$ is isomorphic to the unique irreducible subrepresentation of  $\mathrm{Ind}^{\widetilde{G}_m}_{\widetilde{T}_m^{\rm e}N^{\ast}_m}(\omega^{\psi} \otimes \delta_{B_m}^{-1/4})$ \cite[\S 1E]{Yamana}.

\par
For each character $\varrho$ of $F^{\times}$, we define a genuine character $\zeta^{\psi}_{\varrho}$ of $\widetilde{\mathscr{Z}}_m$ by
\[
  \zeta^{\psi}_{\varrho} \left(\xi {\bf s} \begin{pmatrix} z \vspace{-2ex} && \\ & \ddots& \vspace{-2ex} \\ &&z \end{pmatrix} \right)
  = \xi \varrho(z)\mu^n_{\psi}(z)
\]
for $\xi \in \{\pm 1 \}$ and $z \in F^{\times e(m)}$ with $m=2n$ even or $m=2n+1$ odd. We extend $\theta^{\psi}_{m-1}$ to the representation $\theta^{\psi}_{m-1} \boxtimes \zeta^{\psi}_{\varrho}$ of the semidirect product $(\widetilde{G}_{m-1}\times \widetilde{\mathscr Z}_m) \ltimes U_m$ by letting $\widetilde{G}_{m-1}\times \widetilde{\mathscr Z}_m$ act by $\theta^{\psi}_{m-1} \boxtimes \zeta^{\psi}_{\varrho}$ and letting $U^{\ast}_m$ act trivially. We consider the normalized induced representation
\[
  I_{\psi}(s,\varrho)=\mathrm{Ind}^{\widetilde{G}_m}_{\widetilde{P}_m\widetilde{\mathscr{Z}}_m}(\theta^{\psi}_{m-1} \boxtimes \zeta^{\psi}_{\varrho}) \otimes \delta^{s/4}_{P_{n-1,1}}.
\]

\par
We present the involution $g \mapsto {^{\iota}g}$ of $G_m$ defined by ${^{\iota}g}=w_m {^tg^{-1}} w_m$, where
\[ 
  w_m:=\begin{pmatrix}  \vspace{-2ex} && 1 \\ & \iddots& \vspace{-2ex} \\ 1 && \end{pmatrix}
\]
is the long Weyl element. Kable and Yamana \cite[Proposition 1.3]{Yamana} extend the automorphism of $G_m$ to a lift $\widetilde{g} \mapsto {^{\iota}\widetilde{g}}$ of $\widetilde{G}_m$ satisfying
\[
  {^{\iota}{\bf s}}(a)={\bf s}({^{\iota}a})\prod_{i > j}(a_i,a_j), \quad\quad {^{\iota}\widetilde{z}}=\widetilde{z}^{-1}, \quad\quad {^{\iota}({^{\iota}\widetilde{g})}}=\widetilde{g}, \quad \quad
   {^{\iota}{\bf s}}(n)={\bf s}({^{\iota}n})
\]
for all $a=\mathrm{diag}(a_1,a_2,\dotsm,a_m) \in A_m$, $\widetilde{z} \in \widetilde{\mathscr Z}_m$, $\widetilde{g} \in \widetilde{G}_m$, and $n \in N_m$.
Furthermore if $h : K_m \rightarrow \widetilde{G}_m$ is a homomorphism, then $h$ is compatible with the involution $\iota$, that is to say, $h({^{\iota}k})={^{\iota}h}(k)$ for all $k \in K_m$. We define the normalized operator $N(s,\varrho,\psi) :  I_{\psi}(s,\varrho) \longrightarrow  I_{\psi^{-1}}(-s,\varrho^{-1})$ by
\[
 N(s,\varrho,\psi)(\widetilde{g}):=\gamma(s,\varrho^{-2},\psi)[M(s,\varrho)f_{(s)}](^{\iota}\widetilde{g})
\]
where the unnormalized intertwining operator is given for $\mathrm{Re}(s) \gg 0$ by the integral
\[
  [M(s,\varrho)f_{(s)}](\widetilde{g}):=\int_{F^{m-1}} f_{(s)} \left( {\bf s}  \begin{pmatrix} & I_{m-1} \\ 1 & \end{pmatrix} {\bf s}  \begin{pmatrix} 1 & x \\ &  I_{m-1} \end{pmatrix} \widetilde{g} \right)dx
  \]
and by meromorphic continuation otherwise for $f_{(s)} \in I_{\psi}(s,\varrho)$.

\par 
A $\widetilde{K}_m$-finite function $f : \mathbb{C} \times \widetilde{G}_m \rightarrow \mathbb{C}$ is called a {\it section} if the mapping $\widetilde{g} \mapsto f(s,\widetilde{g})$
belong to $I_{\psi}(s,\varrho)$ for all $s$. A section $f_{(s)} \in I_{\psi}(s,\varrho)$ is said to be a {\it standard (flat) section} if its restriction to $\widetilde{K}_m$ is independent of $s$. We denote by $V_{std}(s,\varrho,\psi)$ the space of standard sections. The space of {\it holomorphic sections} is defined by $V_{hol}(s,\varrho,\psi):=\mathbb{C}[q^{\pm s/4}]\otimes V_{std}(s,\varrho,\psi)$. The elements of $V_{rat}(s,\varrho,\psi):=\mathbb{C}(q^{-s/4})\otimes V_{std}(s,\varrho,\psi)$ is said to be {\it rational sections}. The space of {\it good sections} $V_{good}(2s-1,\varrho,\psi)$ in the sense of Piatetski-Shapiro and Rallis is defined to consist of the followings:
\begin{enumerate}[label=$(\mathrm{\roman*})$]
\item\label{Good-1}  $V_{hol}(2s-1,\varrho,\psi)$
\item\label{Good-2}  $N(1-2s,\varrho^{-1},\psi^{-1})[V_{hol}(1-2s,\varrho^{-1},\psi^{-1})]$.
\end{enumerate}
The normalized operator stabilizes the good section in the sense of 
\[
N(2s-1,\varrho,\psi) [V_{good}(2s-1,\varrho,\psi)] \subseteq V_{good}(1-2s,\varrho^{-1},\psi^{-1})
\]
(cf. \cite[Lemma 3.5]{Yamana}), so this is indeed a good family of sections for normalized intertwining operator. The normalized $\kappa(K_1(\mathfrak{p}^c))$-invariant section (a test function) $f^{K_1(\mathfrak{p}^c)}_{(2s-1)}$ is given in the following fashion (cf. \cite[Proof of Lemma 2.2]{KaYa});
\begin{enumerate}[label=$(\mathrm{\roman*})$]
\item\label{TestFunction-1} When $\pi$ is ramified $(c > 0)$,
\[
f^{K_1(\mathfrak{p}^c)}_{(2s-1)}(\widetilde{g})=
\begin{cases}
\zeta^{\psi}_{\omega^{-1}_{\pi}}(\widetilde{z}) \delta^{\frac{2s+1}{4}}_{P_{m-1,1}}(\widetilde{p}) \theta_{m-1}^{\psi}(\widetilde{p}) & \text{if $\widetilde{g}=\widetilde{z}\widetilde{p}k$,\;\; $\widetilde{z} \in \widetilde{\mathscr{Z}}_m, \widetilde{p} \in \widetilde{P}_{m},k \in \kappa( K_1(\mathfrak{p}^c))$,} \\
0 & \text{if $\widetilde{g} \notin \widetilde{\mathscr{Z}}_m\widetilde{P}_{m} \kappa(K_1(\mathfrak{p}^c))$.}
\end{cases}
\] 
\item\label{TestFunction-2} When $\pi$ is unramified $(c=0)$,
\[
f^{K_m}_{(2s-1)}(\widetilde{g})=
\begin{cases}
L(ms,\omega^2_{\pi}) \zeta^{\psi}_{\omega^{-1}_{\pi}}(\widetilde{z})  \delta^{\frac{2s+1}{4}}_{P_{m-1,1}}(\widetilde{p}) \theta_{m-1}^{\psi}(\widetilde{p}) & \text{if $\widetilde{g}=\widetilde{z}\widetilde{p}k$,\;\; $\widetilde{z} \in \widetilde{\mathscr{Z}}_m, \widetilde{p} \in \widetilde{P}_{m},k \in \kappa(K_m)$}, \\
0 & \text{if $\widetilde{g} \notin \widetilde{\mathscr{Z}}_m\widetilde{P}_{m} \kappa(K_m)$.}
\end{cases}
\]
\end{enumerate}
The following lemma is briefly mentioned in the study of unramified calculation \cite[\S 3H]{Yamana}. We take this occasion to refine this matter thoroughly.

\begin{lemma}
$f^{K_1(\mathfrak{p}^c)}_{(2s-1)}$ is a good section.
\end{lemma}

\begin{proof}
The unitary hypothesis on $\omega_{\pi}$ does not regulate the generality. An arbitrary character $\omega_{\pi}$ is of the form $\omega^{\circ}_{\pi} \nu^t$, where $t \in \mathbb{R}$ and $\omega_{\pi}^{\circ}$ is a unitary character. Since $ I_{\psi}(2s-1+\frac{4t}{m},{\omega^{ \circ}_{\pi}}^{-1}) \otimes \nu^{-\frac{t}{m}} \simeq I_{\psi}(2s-1, {\omega^{ \circ}_{\pi}}^{-1}\nu^{-t})$, we enforce that $\omega_{\pi}$ is unitary at the compensation of shifting $s$ and twisting $\pi$ by unramified character.
With $\pi$ unramified, we have
\[
 L(ms,\omega^2_{\pi}) N(2s-1,\omega^{-1}_{\pi},\psi) f^{K_m}_{2s-1}=L(m-ms,\omega^{-2}_{\pi}) f^{K_m}_{1-2s}.
\]
In particular $L(ms,\omega^2_{\pi}) f^{K_m}_{2s-1}$ is a good section (cf. \cite[Proposition 3.1.(5)]{Yamana14}), because $L(ms,\omega^2_{\pi}) f^{K_m}_{2s-1}$ is holomorphic for $\mathrm{Re}(s) > 0$ and $L(m-ms,\omega^{-2}_{\pi}) f^{K_m}_{1-2s}$ is holomorphic for $\mathrm{Re}(s) < 1/2$. We apply \cite[Proposition 3.11.(4)]{Yamana} to arrive at the conclusion.
When $\pi$ is ramified, $f^{K_1(\mathfrak{p}^c)}_{(2s-1)}$ is just a standard section.
\end{proof}

\par
We define the degenerate character of $N_{m}$ by
\[
\bullet \;\; \psi_{\textbf{e}}(n)=\psi(n_{1,2}+n_{3,4}+ \dotsm + n_{m-1,m}) \quad
\bullet \;\; \psi_{\textbf{o}}(n)=\psi(n_{2,3}+n_{4,5}+\dotsm+n_{m-2,m-1}) 
\]
when $m=2n$ is even. When $m=2n+1$ is odd, we define the degenerate character of $N_{m}$ by
\[
\bullet \;\; \psi_{\textbf{e}}(n)=\psi(n_{2,3}+n_{4,5}+\dotsm+n_{m-1,m})\quad
\bullet \;\; \psi_{\textbf{o}}(n)=\psi(n_{1,2}+n_{3,4}+ \dotsm + n_{m-2,m-1}). 
\]
The unique $\psi^{-1}_{\textbf{e}}$-semi-Whittaker functional corresponds to a $\widetilde{G}_m$-intertwining embedding
$
\lambda : V_{\theta^{\psi}_m} \rightarrow \mathrm{Ind}^{\widetilde{G}_m}_{N^{\ast}_m}(\psi^{-1}_{\textbf{e}})
$ 
satisfying
\[
 \lambda \left[ \theta^{\psi}_m \left( \xi  {\bf s}(n) \widetilde{g}\right) v \right]=\xi \psi^{-1}_{\textbf{e}}(n) \lambda(\theta^{\psi}_m(g)v)
\]
for $n \in N_m$, $\widetilde{g} \in \widetilde{G}_m, \xi \in \{ \pm 1 \}$, and $v \in V_{\theta^{\psi}_m}$ \cite[Proposition 2.5]{Takeda14}. Then we obtain the semi-Whittaker model $\mathcal{W}^{\rm e}(\theta_m^{\psi},\psi_{\bf e}^{-1})$ by setting $(W^{\rm e}_{\theta_m^{\psi}})_v(\widetilde{g}):=\lambda(\theta_m^{\psi}(\widetilde{g})v)$ for $v \in V_{\theta^{\psi}_m}$. For convenience, we will suppress the subscript $v$.

\par
For $W_{\pi} \in \mathcal{W}(\pi,\psi)$, $W^{\rm e}_{\theta_m^{\psi}} \in \mathcal{W}^{\rm e}(\theta_m^{\psi},\psi_{\bf e}^{-1})$ and $f_{(2s-1)} \in V_{good}(2s-1,\omega^{-1}_{\pi},\psi)$, we define the local Bump-Ginzburg integral by
\[
 Z(W_{\pi},W^{\rm e}_{\theta_m^{\psi}},f_{(2s-1)})=\int_{\mathscr{Z}_mN_m \backslash G_m} W_{\pi}(g) W^{\rm e}_{\theta_m^{\psi}}({\bf s}(g))f_{(2s-1)} ({\bf s}(g)) dg.
\]
The integral converges absolutely for $\mathrm{Re}(s) \gg 0$. It is noteworthy that the actual choice of the section $\bf s$ above does not matter, because $W^{\rm e}_{\theta_m^{\psi}}$ and $f_{(2s-1)}$ are both genuine. Hence we omit it from the notation most of time. Let $\mathcal{I}(\pi,\mathrm{Sym}^2)$ be the subspace of $\mathbb{C}(q^{-s/2})$ spanned by local integrals $Z(W_{\pi},W^{\rm e}_{\theta_m^{\psi}},f_{(2s-1)})$, where $W_{\pi} \in \mathcal{W}(\pi,\psi)$, $W^{\rm e}_{\theta_m^{\psi}} \in \mathcal{W}^{\rm e}(\theta_m^{\psi},\psi_{\bf e}^{-1})$ and $f_{(2s-1)} \in V_{good}(2s-1,\omega^{-1}_{\pi},\psi)$. Each $Z(W_{\pi},W^{\rm e}_{\theta_m^{\psi}},f_{(2s-1)})$ is a rational function of $q^{-s/2}$ and hence admits a meromorphic continuation to all of $\mathbb{C}$. The space $\mathcal{I}(\pi,\mathrm{Sym}^2)$ is a $\mathbb{C}[q^{\pm s/2}]$-ideal containing $1$. There exists a normalized generator 
of the form $P(q^{-s/2})^{-1}$ with $P(X) \in \mathbb{C}[X]$ and $P(0)=1$ so that $\mathcal{I}(\pi,\mathrm{Sym}^2)=\left\langle P(q^{-s/2})^{-1} \right\rangle$. We define the local symmetric square $L$-functions by
\[
  L(s,\pi,\mathrm{Sym}^2)=\frac{1}{P(q^{-s/2})}.
\]
As in \cite[p.244]{Yamana}, we expect that $P(q^{-s/2})$ is a polynomial of $q^{-s}$.

\par
Let $\{ \alpha_i \}_{i=1}^r$ denote the Langlands parameter of $\pi$. Combining Proposition \ref{Langlands-RS} with Proposition \ref{formal-exterior}, the formal symmetric square $L$-factor can be paraphrased as
\[
  \mathcal{L}(s,\pi_{ur},\mathrm{Sym}^2)= \prod_{1 \leq i \leq j \leq r} (1-\alpha_i(\varpi)\alpha_j(\varpi)q^{-s})^{-1}.
\]

\begin{remark}
At this moment, the equality $\mathcal{L}(s,\pi_{ur},\mathrm{Sym}^2)=L(s,\pi_{ur},\mathrm{Sym}^2)$ is not known but for $G_2$ \cite{Jo21}. To get around this, we obtain the ``formal symmetric square $L$-factor" $\mathcal{L}(s,\pi_{ur},\mathrm{Sym}^2)$, which is enough at least for describing the main result of our interest. 
\end{remark}

When $m=2n$ is even, we take the Haar measure on $Z^2_m$ so that we assign the volume of $Z^2_m \cap K_1(\mathfrak{p}^c)$ to $1$. The goal is to determine the triple of test vectors which gives rise to the formal $L$-factors $\mathcal{L}(s,\pi_{ur},\mathrm{Sym}^2)$.

\begin{theorem}
Let $\pi \in \mathcal{A}_F(m)$ be a generic representation. We set $c=c(\pi)$. Then we have
\begin{equation}
\label{Sym-Test}
  \mathcal{L}(s,\pi_{ur},\mathrm{Sym}^2)=Z(W^{\circ}_{\pi}, W^{{\rm e}^{\circ}}_{\theta_m^{\psi}}, f^{K_1(\mathfrak{p}^c)}_{(2s-1)}).
\end{equation}
\end{theorem}

\begin{proof}
Although the proof presented seems to be well-known, for lack of a precise reference apart from unramifed cases \cite[\S 3H]{Yamana}, we provide all the detailed accounts. Since the integrand is left $N_m$-invariant and right $K_1(\mathfrak{p}^c)$-invariant, we have
\begin{equation}
\label{sym-Iwasawa}
\begin{split}
  Z(W^{\circ}_{\pi}, W^{{\rm e}^{\circ}}_{\theta_m^{\psi}}, f^{K_1(\mathfrak{p}^c)}_{(2s-1)})
  &=\int_{N_m \backslash P_m} W^{\circ}_{\pi} (p) W^{{\rm e}^{\circ}}_{\theta_m^{\psi}} (p) f^{K_1(\mathfrak{p}^c)}_{(2s-1)}(p) |\mathrm{det}(p)|^{-1} \,dp\\
  &=\int_{N_{m-1} \backslash G_{m-1}} W^{\circ}_{\pi} \begin{pmatrix} g & \\ & 1 \end{pmatrix} W^{{\rm e}^{\circ}}_{\theta_m^{\psi}} \begin{pmatrix} g & \\ & 1 \end{pmatrix} \theta_{m-1}^{\psi}(g)
  |\mathrm{det}(g)|^{\frac{2s+1}{4}-1} \,dg,
  \end{split}
\end{equation}
taking into account the support of $f^{K_1(\mathfrak{p}^c)}_{(2s-1)}$ inside $\mathscr{Z}_mP_mK_1(\mathfrak{p}^c)$. In the aspect of Iwasawa decomposition, this can be expressed as
\[
 Z(W^{\circ}_{\pi}, W^{{\rm e}^{\circ}}_{\theta_m^{\psi}}, f^{K_1(\mathfrak{p}^c)}_{(2s-1)}) 
 =\int_{A_{m-1}} W^{\circ}_{\pi} \begin{pmatrix} a & \\ & 1 \end{pmatrix} W^{{\rm e}^{\circ}}_{\theta_m^{\psi}} \begin{pmatrix} a & \\ & 1 \end{pmatrix} \theta_{m-1}^{\psi}(a) \delta_{B_{m-1}}^{-1}(a)
  |\mathrm{det}(a)|^{\frac{2s+1}{4}-1} \,da.
\]
In view of Theorem \ref{NewV} coupled with \cite[P. 170]{BuGi92}, we may infer that
\[
\begin{split}
 Z(W^{\circ}_{\pi}, W^{{\rm e}^{\circ}}_{\theta_m^{\psi}}, f^{K_1(\mathfrak{p}^c)}_{(2s-1)}) 
& =\int_{A_r}  W^{\circ}_{\pi_{ur}}(a^{\prime}) \nu^{\frac{m-r}{2}}(a^{\prime}) {\bf 1}_{\mathcal{O}}(a_r) 
\delta_{B_{m}}^{\frac{1}{4}}\begin{pmatrix} a^{\prime} & \\ & I_{m-r}\end{pmatrix} 
  \delta_{B_{m-1}}^{\frac{1}{4}-1}\begin{pmatrix} a^{\prime} & \\ & I_{m-r-1}\end{pmatrix}\\
  &\quad \times  |\mathrm{det}(a^{\prime})|^{\frac{2s+1}{4}-1} \,da^{\prime}\\
  &=\int_{A_r}  W^{\circ}_{\pi_{ur}}(a^{\prime}) \nu^{\frac{m-r}{2}}(a^{\prime}) {\bf 1}_{\mathcal{O}}(a_r) \delta_{B_{r}}^{\frac{1}{4}}(a^{\prime}) \nu^{\frac{m-r}{4}}(a^{\prime})
  \delta_{B_{r-1}}^{\frac{1}{4}}(a'') \nu^{\frac{m-r}{4}}(a^{\prime})\nu^{-\frac{r}{4}}(a_r)
  \\
   &\quad \times  \delta_{B_{r}}^{-1}(a^{\prime})  \nu^{r-(m-1)}(a^{\prime}) |\mathrm{det}(a^{\prime})|^{\frac{2s+1}{4}-1} \,da^{\prime},\\
\end{split}
\]
where $a'=(a'',a_r)$. For each $\lambda=(\lambda_1,\dotsm,\lambda_r) \in \mathbb{Z}^r$,  we write $a^{\prime}_{\lambda}={\rm diag}(\varpi^{\lambda_1},\dotsm,\varpi^{\lambda_r})=(a''_{\lambda},a_r)$. We call $\lambda$ even if all components of $\lambda_i$ are even and we let $\alpha=(\alpha_1,\alpha_2,\dotsm,\alpha_r)$ denote the Langlands parameter of $\pi$. Then the integral is equal to
\[
 Z(W^{\circ}_{\pi}, W^{{\rm e}^{\circ}}_{\theta_m^{\psi}}, f^{K_1(\mathfrak{p}^c)}_{(2s-1)})
 =\underset{\lambda_1 \geq \dotsm \geq \lambda_r \geq 0}{\sum_{{\rm even}\;\; \lambda \in \mathbb{Z}^{r}}} \delta^{1/2}_{B_r}(a'_{\lambda}) s_{\lambda}(\alpha)
 \delta^{1/4}_{B_r}(a'_{\lambda}) \delta^{1/4}_{B_{r-1}}(a''_{\lambda})\nu^{-\frac{r}{4}}(a_r)\delta_{B_{r}}^{-1}(a^{\prime}) |\mathrm{det}(a_{\lambda}^{\prime})|^{\frac{2s+1}{4}}, 
\]
where we use Shintani formula and  $s_{\lambda}$ denotes the Schur polynomial defined in \cite[Lemma 3.1]{BKL}. Observing that $\delta_{B_r}=\delta_{B_{r-1}}\delta_{P_{r-1,1}}$, this clearly justifies that
\[
 Z(W^{\circ}_{\pi}, W^{{\rm e}^{\circ}}_{\theta_m^{\psi}}, f^{K_1(\mathfrak{p}^c)}_{(2s-1)})
 =\underset{\lambda_1 \geq \dotsm \geq \lambda_r \geq 0}{\sum_{{\rm even}\;\; \lambda \in \mathbb{Z}^{r}}} |\mathrm{det}(a_{\lambda}^{\prime})|^{\frac{s}{2}} s_{\lambda}(\alpha).
 \]
 The remaining part of the proof continues as in \cite[Theorem 4.1]{BuGi92} and \citelist{\cite{Takeda14}*{Proposition 3.16} \cite{KaYa}*{Proposition 3.4}}
\end{proof}

The equality of $\mathcal{L}(s,\pi_{ur},\mathrm{Sym}^2)$ and $L(s,\pi,\mathrm{Sym}^2)$ fails in general.

\begin{corollary} We have
\[ 
\mathcal{L}(s,\pi_{ur},\mathrm{Sym}^2)=P(q^{-s/2})L(s,\pi,\mathrm{Sym}^2)
\]
for a polynomial $P(X) \in \mathbb{C}[X]$ satisfying $P(0)=1$.
\end{corollary}

Let us now turn our attention to $G_m$-invariant trilinear forms and Bump-Ginzburg period integrals.

\begin{proposition}\cite[Theorem 2.14.(2)]{Yamana}
Let $\pi \in \mathcal{A}_F(m)$ be a unitary generic representations. For any $W_{\pi} \in \mathcal{W}(\pi,\psi)$, $W^{\rm e}_{\theta_m^{\psi}} \in \mathcal{W}^{\rm e}(\theta_m^{\psi},\psi_{\bf e}^{-1})$, and $f_{(1)} \in V_{good}(1,\omega^{-1}_{\pi},\psi)$, the integral
\[
T(W_{\pi},W^{{\rm e}}_{\theta_m^{\psi}}, f_{(1)}):=\int_{\mathscr{Z}_mN_m \backslash G_m} W_{\pi} (g) W^{\rm e}_{\theta_m^{\psi}} (g) f_{(1)}(g) dg
  \]
converges absolutely and defines a $G_m$-trilinear functional on the space $\mathcal{W}(\pi,\psi) \otimes \mathcal{W}^{\rm e}(\theta_m^{\psi},\psi_{\bf e}^{-1}) \otimes V_{good}(1,\omega^{-1}_{\pi},\psi)$.
\end{proposition}

The $G_m$-trilinear form $T(W_{\pi},W^{{\rm e}}_{\theta_m^{\psi}}, f_{(1)})$ is not actually vanished.

\begin{theorem}
\label{unitary-distinguished}
Let $\pi \in \mathcal{A}_F(m)$ be a unitary generic representation. We set $c=c(\pi)$. Then we have
\[
  \int_{\mathscr{Z}_mN_m \backslash G_m} W^{\circ}_{\pi} (g) W^{{\rm e}^{\circ}}_{\theta_m^{\psi}} (g) f^{K_1(\mathfrak{p}^c)}_{(1)}(g) dg=\mathcal{L}(1,\pi_{ur},\mathrm{Sym}^2).
  \]
\end{theorem}

\begin{proof}
With help of the identity $L(s,\pi_{ur} \times \pi_{ur})=L(s,\pi_{ur},\wedge^2)\mathcal{L}(s,\pi_{ur},\mathrm{Sym}^2)$ along with \cite[Proposition 3.17]{JaSh81}, $\mathcal{L}(s,\pi_{ur},\mathrm{Sym}^2)$ does not afford any poles at $s=1$. This can be achieved by evaluating at $s=1$ on the both sides of \eqref{Sym-Test}.
\end{proof}

We say that a representation $\pi$ of $G_m$ is $\theta$-{\it distinguished}, if $\mathrm{Hom}_{G_m}(\pi \otimes \theta^{\psi}_m \otimes \theta^{{\psi}^{-1}}_m,{\bf 1}_{G_m}) \neq 0$.

\begin{proposition}
\label{Yamana-mulone}
 Let $\pi \in \mathcal{A}_F(m)$ be a unitary representation. Then
\[
  {\rm dim}_{\mathbb{C}} {\rm Hom}_{G_m}( \pi \otimes \theta_m^{\psi}  \otimes I_{\psi}(1,\omega^{-1}_{\pi}),{\bf 1}_{G_m}   ) \leq 1.
\]
The equality holds when $\pi$ is generic. In particular if $\pi$ is $\theta$-distinguished, then $(W_{\pi},W^{{\rm e}}_{\theta_m^{\psi}}, f_{(1)}) \mapsto T(W_{\pi},W^{{\rm e}}_{\theta_m^{\psi}}, f_{(1)})$ gives a non-trivial trilinear form belonging to $\mathrm{Hom}_{G_m}(\pi \otimes \theta^{\psi}_m \otimes \theta^{{\psi}^{-1}}_m,{\bf 1}_{G_m})$.

\end{proposition}

\begin{proof}
The first assertion is simply a statement from \cite[Theorem 3.8.(5)]{Yamana}. In light of Theorem \ref{unitary-distinguished}, the injection map  (see \cite[(2.3)]{Yamana})
 \[
 \mathrm{Hom}_{G_m}(\pi \otimes \theta^{\psi}_m \otimes \theta^{{\psi}^{-1}}_m,{\bf 1}_{G_m}) \hookrightarrow 
 \mathrm{Hom}_{G_m}(\pi \otimes \theta^{\psi}_m \otimes I_{\psi}(1,\omega^{-1}_{\pi}),{\bf 1}_{G_m})
 \]
 becomes the isomorphism
 \[
 \mathrm{Hom}_{G_m}(\pi \otimes \theta^{\psi}_m \otimes \theta^{{\psi}^{-1}}_m,{\bf 1}_{G_m}) \simeq
 \mathrm{Hom}_{G_m}(\pi \otimes \theta^{\psi}_m \otimes I_{\psi}(1,\omega^{-1}_{\pi}),{\bf 1}_{G_m}).
 \]
 Afterwards, $(W_{\pi},W^{\rm e}_{\theta_m^{\psi}},f_{(1)}) \mapsto T(W_{\pi},W^{\rm e}_{\theta_m^{\psi}},f_{(1)})$ gives a non-trivial $G_{m}$-invariant trilinear form in the space $\mathrm{Hom}_{G_m}(\pi \otimes \theta^{\psi}_m \otimes \theta^{{\psi}^{-1}}_m,{\bf 1}_{G_m})$.
\end{proof}

We close this section with two remarks which were shortly mentioned in \S \ref{Intro}. Introduction.

\begin{remark}
\label{rmk1}
In principle the Frobenius reciprocity (cf. \cite[Proof of Theorem 2.14]{Yamana}) demonstrates the isomorphism of the space
\[
 {\rm Hom}_{G_m}( \pi \otimes \theta_m^{\psi}  \otimes I_{\psi}(1,\omega^{-1}_{\pi}) ,{\bf 1}_{G_m} ) \simeq
  {\rm Hom}_{P_m}( \pi|_{P_m} \otimes \theta_m^{\psi}|_{P_m} \otimes \Psi^+(\theta_{m-1}^{\psi}) \otimes \nu^{-3/4},{\bf 1}_{G_m} ).
\]
Knowing that $\theta_m^{\psi}$ affords Kirillov model $\theta_m^{\psi}|_{P_m} \simeq \mathcal{W}^{\rm e}(\theta_m^{\psi},\psi_{\bf e}^{-1})|_{P_m}$ \cite[Theorem 4.3]{Kable},
this isomorphism can be phrased in terms of integral representations. Performing the same steps repeatedly as in \eqref{sym-Iwasawa}, the integration over $K_n$ can be absorbed to get
\begin{equation}
\label{mirabolic-theta-distinct}
\int_{N_m \backslash P_m}  W^{\circ}_{\pi} (p) W^{{\rm e}^{\circ}}_{\theta_m^{\psi}} (p) f^{K_1(\mathfrak{p}^c)}_{(1)}(p) |\mathrm{det}(p)|^{-1} dp
 =\begin{cases}
    \mathcal{L}(1,\pi_{ur},\mathrm{Sym}^2)
   & \text{if $\pi$ is ramified,} \\
 \dfrac{ \mathcal{L}(1,\pi,\mathrm{Sym}^2)   }{L(m,\omega^2_{\pi})}  & \text{if $\pi$ is unramified,} 
  \end{cases}
 \end{equation}
 which is easily seen to be equal to
 \[
 \int_{N_{m-1} \backslash G_{m-1}} W^{\circ}_{\pi}  \begin{pmatrix} g & \\ & 1 \end{pmatrix}  W^{{\rm e}^{\circ}}_{\theta_m^{\psi}} \begin{pmatrix} g & \\ & 1 \end{pmatrix}
(\theta_{m-1}^{\psi} \otimes \nu^{1/2})(g) \nu^{-3/4}(g) dg.
 \]
 The last integral is evidently an element of ${\rm Hom}_{P_m}( \pi|_{P_m} \otimes \theta_m^{\psi}|_{P_m} \otimes \Psi^+(\theta_{m-1}^{\psi}) \otimes \nu^{-3/4},{\bf 1}_{G_m} )$ and the expression \eqref{mirabolic-theta-distinct} is parallel to what we have went through before; Theorem \ref{RS-priods}, Theorem \ref{unitary-flicker}, Theorem \ref{Shalika-unitary}, and Theorem \ref{Hm-unitary}.
 \end{remark}

\begin{remark} 
\label{rmk2}
Let $\pi \in \mathcal{A}_F(m)$ be a generic representation which is $\theta$-distinguished. By \cite[Corollary 4.19]{Kaplan17}, it follows that $\pi$ is self dual. Owing to Proposition \ref{Bernstein-s=1}, we find that $L(s,\pi \times \pi)$ is holomorphic at $s=1$. At this point, we fail to formulate the agreement $\mathcal{L}(s,\pi, \mathrm{Sym}^2)=L(s,\pi, \mathrm{Sym}^2)$, which implies the identity
\[
L(s,\pi \times \pi)=L(s,\pi, \wedge^2)L(s,\pi, \mathrm{Sym}^2)
\]
with $\mathcal{L}(s,\pi, \mathrm{Sym}^2)$ appearing in \eqref{Ext-Sym-identity}. Nevertheless the statement has been recently confirmed for $m=2$ \cite{Jo21}. Taking it for granted, it can be seen that $L(s,\pi,\mathrm{Sym}^2)$ is holomorphic at $s=1$. 
Evaluating at $s=1$, both sides of \eqref{Sym-Test} read
\[
  \int_{\mathscr{Z}_mN_m \backslash G_m} W^{\circ}_{\pi} (g) W^{{\rm e}^{\circ}}_{\theta_m^{\psi}} (g) f^{K_1(\mathfrak{p}^c)}_{(1)}(g) dg=L(1,\pi_{ur},\mathrm{Sym}^2).
  \]
  Any $\theta$-distinguished representation always satisfies $\omega^2_{\pi}={\bf 1}_{F^{\times}}$ \cite[\S 4]{Kaplan17}. With this said, \eqref{mirabolic-theta-distinct} is reshaped as follows;
  \[
\int_{N_m \backslash P_m}  W^{\circ}_{\pi} (p) W^{{\rm e}^{\circ}}_{\theta_m^{\psi}} (p) f^{K_1(\mathfrak{p}^c)}_{(1)}(p) |\mathrm{det}(p)|^{-1} dp
 =\begin{cases}
    L(1,\pi_{ur},\mathrm{Sym}^2)
   & \text{if $\pi$ is ramified,} \\
 \dfrac{ L(1,\pi,\mathrm{Sym}^2)   }{L(m,{\bf 1}_{F^{\times}})}  & \text{otherwise.} 
  \end{cases}
 \]
\end{remark}

\begin{acknowledgments}
This project is inspired by the response to a question raised by James Cogdell as to whether the space of $(P_{2n}\cap S_{2n},\Theta)$-invariant Shalika functionals in \cite[Lemma 3.2]{Jo20} is trivial or not. The author is indebted to J. Cogdell for drawing the author's attention to this problem. We would like to thank Muthu Krishnamurthy for kindly explaining their joint work \cite{BKL} to me, encouraging me to write this manuscript, and giving many invaluable comments over the year. We are also grateful to Peter Humphries for many helpful suggestions on earlier versions of this article. 
Finally, we express our sincere appreciation to the referee for a number of constructive comments, which significantly improves exposition of literature in the paper.  
\end{acknowledgments}



 \bibliographystyle{amsplain}

\end{document}